\documentclass[hidelinks,onefignum,onetabnum]{siamart220329}

\usepackage{lipsum}
\usepackage{svg,subcaption}
\usepackage{amsfonts}
\usepackage{graphicx}
\usepackage{epstopdf}
\usepackage{algorithmic}
\ifpdf
  \DeclareGraphicsExtensions{.eps,.pdf,.png,.jpg}
\else
  \DeclareGraphicsExtensions{.eps}
\fi

\newsiamremark{remark}{Remark}
\newsiamremark{hypothesis}{Hypothesis}
\crefname{hypothesis}{Hypothesis}{Hypotheses}
\newsiamthm{claim}{Claim}

\usepackage{graphicx}
\usepackage[utf8]{inputenc}
\usepackage{mathtools}
\usepackage{mathrsfs}
\usepackage{amsmath, amssymb, amsfonts, color}
\usepackage{hyperref}
\usepackage{comment}
\usepackage{etoolbox}

\DeclareFontFamily{U}{matha}{\hyphenchar\font45}
\DeclareFontShape{U}{matha}{m}{n}{
<-6> matha5 <6-7> matha6 <7-8> matha7
<8-9> matha8 <9-10> matha9
<10-12> matha10 <12-> matha12
}{}
\DeclareSymbolFont{matha}{U}{matha}{m}{n}

\DeclareFontFamily{U}{mathx}{\hyphenchar\font45}
\DeclareFontShape{U}{mathx}{m}{n}{
<-6> mathx5 <6-7> mathx6 <7-8> mathx7
<8-9> mathx8 <9-10> mathx9
<10-12> mathx10 <12-> mathx12
}{}
\DeclareSymbolFont{mathx}{U}{mathx}{m}{n}

\DeclareMathDelimiter{\vvvert} {0}{matha}{"7E}{mathx}{"17}%

\DeclarePairedDelimiterX{\normiii}[1]
{\vvvert}
{\vvvert}
{\ifblank{#1}{\:\cdot\:}{#1}}
\newcommand{\A}{\mathbb{A}}

\newcommand{\I}{\mathbb{I}}

\newcommand{\W}{\mathcal{W}}
\newcommand{\LL}{\mathcal{L}}
\newcommand{\R}{\mathbb{R}}

\newcommand{\half}{\frac{1}{2}}
\newcommand{\norm}[1]{\left\lVert#1\right\rVert}
\newcommand{\abs}[1]{\left\lvert#1\right\rvert}

\newtheorem{thm}{Theorem} 

\newtheorem{defn}{Definition}

\newtheorem{propn}{Proposition}

\newcommand{\cc}{\mathbf{c}}
\newcommand{\ee}{\mathbf{e}}
\newcommand{\xx}{\mathbf{x}}
\newcommand{\yy}{\mathbf{y}}

\newcommand{\ff}{\mathbf{f}}

\newcommand{\bb}{\mathbf{b}}

\newcommand{\uu}{\mathbf{u}}

\newcommand{\vv}{\mathbf{v}}
\newcommand{\ww}{\mathbf{w}}
\newcommand{\zz}{\mathbf{z}}

\newcommand{\QQ}{\mathbf{Q}}

\headers{Quasistatic Evolution with Unstable Forces}{D. Bhattacharya and R. P. Lipton}

\title{Quasistatic Evolution with Unstable Forces\thanks{\textbf{Funding:} This material is based upon work supported by the U. S. Army Research Laboratory and the U. S. Army Research Office under Contract/Grant Number W911NF-19-1-0245.}
}

\author{Debdeep Bhattacharya\thanks{
 Department of Mathematics,
Louisiana State University,
Baton Rouge, LA 70803,
USA (\email{debdeepbh@lsu.edu}).}
\and Robert P. Lipton\thanks{
 Department of Mathematics,
 LSU Center of Computation \& Technology,
Louisiana State University,
Baton Rouge, LA 70803, USA
  (\email{lipton@lsu.edu}).}
}

\usepackage{amsopn}

\ifpdf
\hypersetup{
  pdftitle={Quasistatic Evolution with Unstable Forces},
  pdfauthor={D. Bhattacharya and R. P. Lipton}
}
\fi

\begin{document}

\maketitle

\begin{abstract}
We consider load controlled quasistatic evolution. Well posedness results for the nonlocal continuum model related to peridynamics are established.  We show local existence and uniqueness of quasistatic evolution for load paths originating at stable critical points. These points can be associated with local energy minima
among the convex set of deformations belonging to the strength domain of the material. The evolution of the displacements however is not constrained to lie inside the strength domain of the material. The load-controlled evolution is shown to exhibit energy balance. 
\end{abstract}



\section{Introduction}%
\label{sec:introduction}

The model studied here is a nonlocal continuum model where the length scale of force interaction between points is taken to be at least an order of magnitude smaller than the characteristic length of the domain. 
We pose a nonlinear and nonlocal field theory of peridynamic (PD) type \cite{SILLING2000175}. The unknown is the displacement field $\uu$ at a point $\xx$ in the body at time $t$. 
For PD models the force interaction occurs between a point $\xx$ and another point $\yy$ when $\yy$ lies within a sphere $H_\epsilon(\xx)$ of radius $\epsilon$ centered at $\xx$. The force between $\xx$ and $\yy$, is determined by the displacement at each point through a constitutive law. The net force on $\xx$ is the force averaged over all $\yy$ in the sphere. The radius $\epsilon$ is often referred to as the horizon. The net force is referred to as the PD force of the body acting on $\xx$. The constitutive law used here is of cohesive type; the force between two points initially increases with strain until a maximum force is reached and then the force decreases to zero with a continued increase in strain.

The objective of PD field theory is to account for elastic interaction where the material is intact as well as the emergence and propagation of failure zones characterized by vanishing force. PD models are inherently multiscale, coupling fracture caused by breaking bonds at the atomic scale with elastic deformation at the macroscopic scale.
Dynamic simulations for cohesive PD show that failure zones are naturally localized by the model and appear as thin and crack like \cite{lipton2019complex,jhalipton2020}.  Both localization and emergent behavior is the hallmark of simulations using the PD formulation introduced in \cite{SILLING2000175,silling2007peridynamic}, see for example \cite{silling2005meshfree,bobaru2010peridynamic,oterkus2014fully,Kilic2009,silling2017modeling,madenci2017ordinary,parksyueyoutrask} where different PD models are developed and applied. There is a growing mathematical theory supporting well posedness of the nonlocal fracture modeling and numerical simulations for different PD fracture  models \cite{emmrichpuhst2016,dutian2018,lipton2018free,jha2018numerical,parksyueyoutrask}. Cohesive PD models like the one considered here are seen to recover classic Griffith fracture energy from the PD energy in the limit $\epsilon\rightarrow 0$ through $\Gamma$ convergence,
\cite{lipton2014horizonlimit,lipton2016cohesive,lipton2019complex}.  Convergence of PD fields adjacent to crack like defects to elastic fields with zero normal traction on classic cracks as well as the classic kinetic relations for crack growth (see \cite{Freund,ravi2004dynamic,anderson2017fracture}) are recovered in the $\epsilon=0$ limit \cite{liptonjha2021,{jhalipton2020}}. 

In the absence of inertia one considers quasistatic or rate independent evolution. Nonlocal formulations of rate independent linearized plastic evolution is formulated and solved \cite{kruvzik2018quasistatic}. There it is shown that the nonlocal peridynamic model $\Gamma$ converges to classic local elastoplasticity as the interaction range goes to zero. Nonlocal equilibrium problems for linear elasticity are  shown to $\Gamma$ converge to classic elastic boundary value problems \cite{mengesha2015variational}.
Nonlocal multiscale peridynamic models have been shown to rigorously to recover classic strongly coupled local theories of elasticity with highly oscillatory coefficients \cite{scott2020asymptotic}. Memory effects are recovered by two scale limits of highly oscillatory linear nonlocal models using two-scale homogenization \cite{alali2012multiscale,du2016multiscale}.


This article investigates the quasi-static regime using the constitutive model of cohesive PD. We consider the zero inertia limit 
and focus on {the rate independent} evolution governed by the PD equations in the absence of acceleration. 
The approach to evolution taken here is a departure from the quasistatic evolution of global energy mininimizers. 
{
Given a body $\Omega\in\R^d$, $d=2,\,3$ we assume at $t=0$ and for $\xx\in\Omega$ that there exists a displacement $\uu_0=\uu(\xx,0)$ and body force $\bb(\xx,0)$ for which we have force balance
\begin{equation}\label{eq:pdQuasi0}
\begin{aligned}
   \mathcal{L}[\uu_0](\xx) = \bb(\xx,0),
\end{aligned}
\end{equation}
where $\mathcal{L}[\uu_0](\xx)$ is the force at $\xx$ due to the deformation $\uu_0$ and given by \eqref{eq: force}.
The quasistatic evolution $\uu(t)=\uu(\xx,t): t\in  [0,T]$ is given by
\begin{equation}\label{eq:pdQuasi1}
\begin{aligned}
   \mathcal{L}[\uu(t)](\xx) = \bb(\xx,t).
\end{aligned}
\end{equation}
for a prescribed load path $\bb(\xx,t): t\in [0,T] \to \R^2$. Here $t$ appears as a load parameter and the load path is independent of parameterization. 
We show existence of an evolution in a neighborhood of $(\uu_0,\bb_0)$ provided the inverse of the Fr\'echet  derivative of  $\LL[\uu(t)]$ exists at $\uu_0$, see theorem \ref{thm:exitence-of-perturbation}. When the load $\bb(t)$ is smooth and the Fr\'echet derivative of $\LL[\uu(t)]$ is invertible for $t\in [0,T]$ then the quasistatic evolution satisfies energy balance, see theorem \ref{thm:Energy-load} of section \ref{sec:main results}.}

{A displacement $\uu(\xx,t)$ is said to be in the strength domain if its strain 
$$S(\yy,\xx)=\frac{\uu(\yy)-\uu(\xx)}{|\yy-\xx|}\cdot\frac{\yy-\xx}{|\yy-\xx|},$$
increases with increasing force for every $\yy$ in the neighborhood of every point in the body, see Definition \ref{strength defn}. If there is a set of points for which force begins to decrease inside their neighborhoods then the deformation lies outside the strength domain. The collection of such points are called softening zones. The key feature is that  the force $\LL[\uu(t)](\xx)$ remains  defined for all points of the body. Here material damage is represented by the softening zone and its propagation is part of the quasistatic evolution.}

{Next we consider the case when $\bb_0$ is such that the solution $\uu_0$ of \eqref{eq:pdQuasi0} lies inside the strength domain. For this case we introduce a stability tensor field defined at every point $\xx$ in the body, see Definition \ref{stabilitytensorload} of section \ref{sec:main results}. We discover that the stability tensor is positive definite when the domain satisfies an interior cone condition and the deformation field lies inside the strength domain of the material, see theorem \ref{thm:stabinvI}
of section \ref{sec:main results}. The interior cone condition automatically excludes sharp domains see figure \ref{fig:int-cone}.
For this case we can show that the Fr\'echet derivative of $\LL[\uu_0]$ is invertible and conclude that a quasistatic evolution exists in an neighborhood of $(\bb_0,\uu_0$), see the discussion below theorem \ref{thm:stabinvI}.}


We relate the existence of a local evolution to the existence of a local minimizer of a PD energy. We show that a quasi-static PD  evolution exists within a neighborhood of a local energy minimizer. Here the minimizing displacement is a local minimizer among fields belonging to the strength domain of the material, see Theorem \ref{thm:Exist-hard loading} of section \ref{sec:main results}. {If instead the displacement $\uu_0$ does not lie in the strength domain we present more general sufficient conditions for invertability given in Theorem \ref{esistence of inverse in strength domain}. Part of the sufficient conditions require that all eigenvalues of $\mathbb{A}[\uu]$ lie outside an open interval containing $0$.}
{We present a necessary condition of invertability of the Fr\'echet derivative of $\LL[\uu_0]$ in terms of the stability tensor that shows that $\mathbb{A}[\uu]$ can not have a zero eigenvalue on a subset of finite measure on $\Omega$, see Theorem \ref{thm:ness} of section \ref{sec:main results}. }  The stability tensor is defined here for the quasistatic case but agrees in form with the stability tensor introduced for dynamic fracture in \cite{sillingweknerascaribobaru} and in  \cite{lipton2014horizonlimit,lipton2019complex}. This tensor appears again for elasto-dynamic  problems  and in energy minimization for equilibrium problems in \cite{DuGunLehZho,MengeshaDuNonlocal14} and elasto-static problems with sign changing kernel \cite{MengeshaDu}. 
{Much of the theory developed here provides the foundation for the quasistatic fracture theory developed and implemented in the sequel \cite{BhattacharyaLiptonDiehl}.}

 The PD model is described in section \ref{sec:background} and the main results are provided in section \ref{sec:main results}. 
 The proofs of all theorems are provided in sections \ref{sec:existence theory for soft} through \ref{sec:necessary}. The article concludes with a summary of results.

\section{Background and problem formulation}
\label{sec:background}
In this paper both two and three dimensional nonlocal formulations are considered.
For dynamics the deformation field $\uu(\xx, t)$ of a material point $\xx$ in the domain $\Omega \subset \R^d,\,\,d=2,3$, at time $t$ satisfies the momentum-balance equation
\begin{align}
\label{balance}
    \rho \ddot{\uu}(\xx, t) =  \int\limits_{H_\epsilon(\xx) \cap \Omega}^{} \ff(\yy, \xx, \uu, t) d\yy + \bb(\xx, t),\ \xx \in \Omega 
\end{align}	
where $\rho$ is the material density, $H_\epsilon(\xx)$ is a sphere centered at the point $\xx$ with radius $\epsilon$ and $\bb$ is the body force density. Here, $H_\epsilon(\xx)\cap \Omega$ is the set of neighboring points $\yy$ in $\Omega$ that can interact with $\xx$. The parameter $\epsilon$ is called the horizon. The force interaction between $\xx$ and $\yy$ is mediated by 
$\ff( \yy, \xx, \uu,t)$ that denotes the force density exerted by a point $\yy \in H_\epsilon(\xx)$  on $\xx$ and is taken to be a function of internal displacement at time $t$. 

When the effect of inertia can be ignored \eqref{balance} reduces to the quasi-static formulation given by
\begin{align}
\label{eq:pdQuasi}
    \int\limits_{H_\epsilon(\xx) \cap \Omega}^{} \ff (\yy, \xx, \uu) d\yy + \bb(\xx,t) = 0.
\end{align}

We introduce the 
nonlocal strain $S(\yy, \xx, \uu)$ between the point $\xx$ and any point $\yy \in H_\epsilon(\xx)$  given by
\begin{align}\label{strain}
    S(\yy, \xx, \uu) = \frac{\uu(\yy) - \uu(\xx)}{\abs{\yy - \xx}} \cdot \ee_{\yy - \xx},
\end{align}	
where $\ee_{\yy- \xx}$ is the unit vector given by
\begin{align*}
\ee_{\yy - \xx}  = \frac{\yy - \xx}{\abs{\yy - \xx}}.
\end{align*}	
Force is related to strain using the constitutive relation given by the cohesive force law  \cite{lipton2014horizonlimit,lipton2016cohesive}. Under this law the force is linear for small strains and for larger strains the force begins to soften and then approaches zero
after reaching a critical strain. 
The nonlocal force density $\ff$ is given in terms of the nonlocal potential $\W(S)$ by
\begin{align}\label{contsit1}
    \ff(\yy, \xx, \uu) = 2 \partial_S \W(S(\yy, \xx, \uu)) \ee_{\yy - \xx},
\end{align}	
where
\begin{align}\label{contsit2}
    \W(S(\yy, \xx, \uu)) = \frac{J^\epsilon(\abs{\yy - \xx})}{\epsilon^{d+1}\omega_d \abs{\yy - \xx}}  g(\sqrt{ \abs{\yy - \xx}} S(\yy, \xx, \uu)).
\end{align}	
Here, $J^\epsilon(r) = J(\frac{r}{\epsilon})$, where $J$ is a non-negative bounded function supported on $[0,1]$. $J$ is called the \textit{influence function} as it determines the influence of the bond force of peridynamic neighbors $\yy$ on the center $\xx$ of $H_\epsilon(\xx)$. The volume of unit ball in $\R^d$ is denoted by $\omega_d$. 
As figure \ref{ConvexConcavea} illustrates we assume that $g(r)$ and the derivatives $g'(r)$, $g''(r)$, and $g'''(r)$ are bounded for $-\infty < r<\infty$. It is required is that $g(0)=0$ and $g(r)>0$ otherwise, $g(r)$ together with its first three derivatives must be bounded, and that $g$ be  convex in the interval $r^e<0<r^c$ and concave outside this interval with finite limits $\lim_{r\rightarrow-\infty}{g(r)}=C^-$ and $\lim_{r\rightarrow\infty}{g(r)}=C^+$. Additionally $\max\{|g''(r|)\}=g''(0)$.
The quasistatic peridynamic equation \eqref{eq:pdQuasi} is expressed by
\begin{align*}
    \LL[\uu](\xx,t) = \bb(\xx,t), \ \xx \in \Omega,
\end{align*}	
where the integral operator $\LL$ is defined as
\begin{align}
      \label{eq: force}
    \LL[\uu](\xx,t) = -\int\limits_{H_\epsilon(\xx) \cap \Omega}^{} {2}\frac{J^\epsilon(\abs{\yy - \xx} )}{\epsilon^{d+1} \omega_d \sqrt{\abs{\yy - \xx}}} g'\left(\sqrt{ \abs{\yy - \xx}} S(\yy, \xx, \uu)\right) \ee_{\yy - \xx} d\yy.
\end{align}	

For this model the strength domain of the material is simple and can be described in terms of the strain $S(\yy,\xx,\uu)$.
\begin{defn}{Strength Domain.}\label{strength defn} 
For $\xx\in\Omega$ and $\yy\in \Omega\cap H_\epsilon(\xx)$ the strength domain of the material is given by all displacements with strain inside the interval
\begin{equation}\label{strength1}
{r^e}<\sqrt{|\yy-\xx|}S(\yy,\xx,\uu)<r^c,
\end{equation}
or
\begin{equation}\label{strength}
\frac{r^e}{\sqrt{|\yy-\xx|}}<S(\yy,\xx,\uu)<\frac{r^c}{\sqrt{|\yy-\xx|}}.
\end{equation}
\end{defn}
This is the set of strains where the magnitude of force increases with increasing strain. Material failure occurs for strains outside this interval where the force becomes unstable. Moreover $g''(\sqrt{|\yy-\xx|}S(\yy,\xx,\uu))>0$ in the strength domain. The strength domain is a convex set. A displacement is said to lie strictly inside the strength domain if $\sqrt{|\yy-\xx|}S(\yy,\xx,\uu)$ lies within a closed interval inside $(r^e,r^c)$.


{Any Lipschitz continuous function with modulus of continuity less than $\omega=\min\{|r^e|,|r^c|\}$ lies within the strength domain provided the horizon for the material is less than one, i.e., $\epsilon<1$
}

\begin{figure}
    \centering
\includegraphics[width=0.8\linewidth]{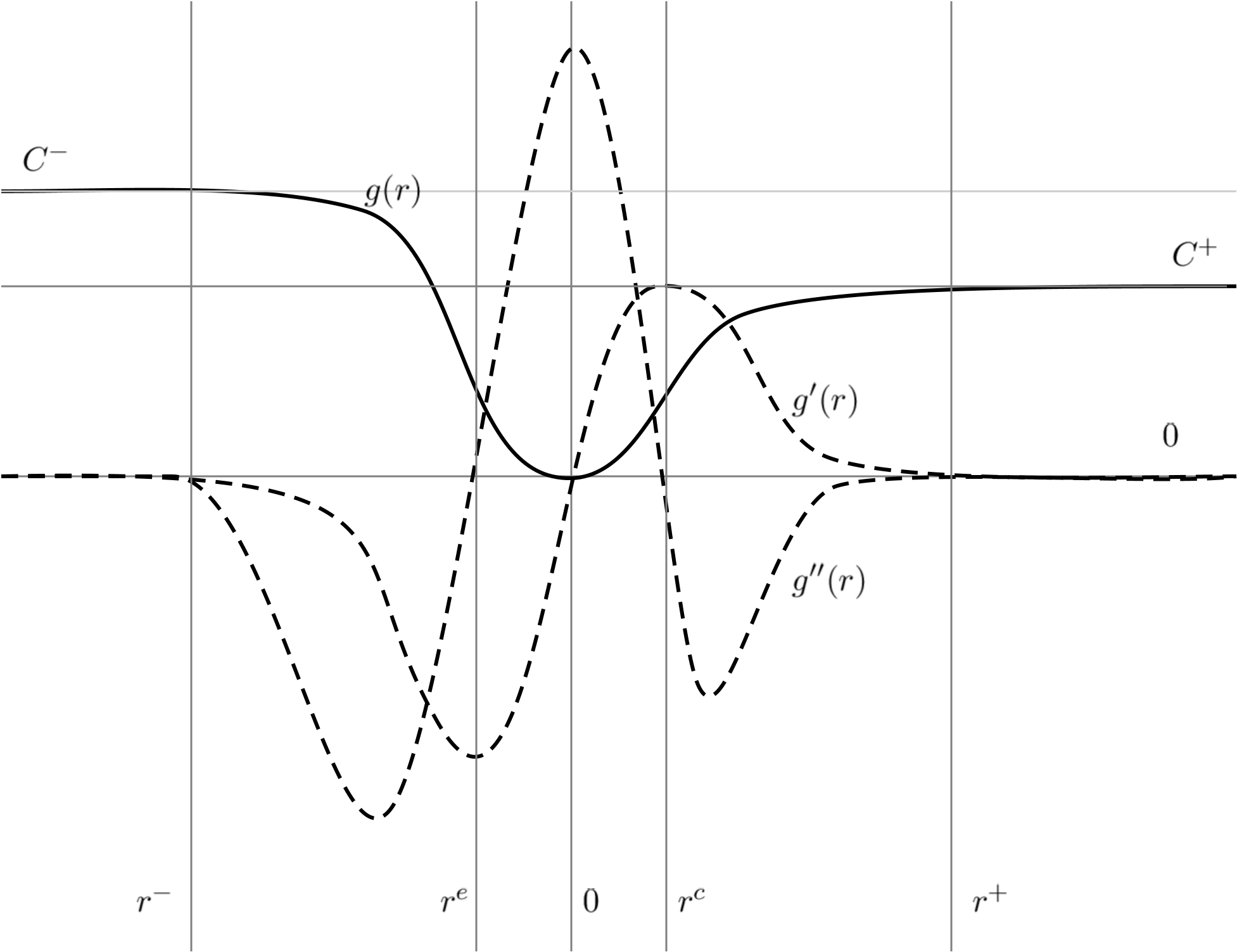}
		  \caption{The double well potential function $g(r)$ and derivatives $g'(r)$ and $g''(r)$  for tensile force. Here $C^+$ and $C^-$ are the asymptotic values of $g$. The derivative of the force potential goes smoothly to zero at $r^+$ and $r^-$.}
		  \label{ConvexConcavea}
\end{figure}

The quasistatic evolution problem is now described. The load path is given by a prescribed body force density $\bb(\xx,t)$ parameterized by $t$, with $0\leq t\leq T$ and written $\bb(t)$. The associated displacement is written $\uu(t):=\uu(\xx,t)$. 
The application of  $\bb$  in the absence of prescribed boundary displacement is referred to as load control \cite{anderson2017fracture}. We say that the displacement $\uu(t)$ satisfies the quasi-static evolution problem for load control with prescribed load path $\bb(t)$, $t \in [0,T]$ if it satisfies
\begin{equation}
    \label{eq:equlib2}
    \LL[\uu(t)] = \bb(t),
\end{equation}
for $0\leq t \leq T$. 

We conclude this section noting that in the peridynamic taxonomy our cohesive model is classified as a bond-based or ordinary state based peridynmic material model outlined in \cite{silling2007peridynamic}. To see this write
\begin{align}\label{eq:pdforce}
\mathcal{L}(\uu)=\int_{\Omega\cup\mathcal{H}_{\epsilon}(\xx)} \left({\mathbf{T}}(\xx)(\yy-\xx)-\mathbf{T}(\yy)(\xx-\yy)\right)\;d\yy,
\end{align}
where
\begin{align}\label{eq:pdstate}
\mathbf{T}(\yy)(\yy-\xx)=\partial_S\mathcal{W}^\epsilon(S(\yy,\xx,\uu(t)))\ee_{\yy-\xx},  \mathbf{T}(\xx)(\xx-\yy)=\partial_S\mathcal{W}^\epsilon(S(\xx,\yy,\uu(t)))\ee_{\xx-\yy}.
\end{align}

\section{Existence and energy balance for load control}
\label{sec:main results}
The section contains the main results and describes the existence and uniqueness of quasistatic evolution within a neighborhood of a prescribed initial deformation of the material. 
In what follows fixed point methods are applied to find  solutions to the quasistatic evolution.
Solutions are elements of a subspace of the well known Lebesgue space $L^\infty(\Omega;\mathbb{R}^d)$ defined by all essentially bounded, measurable displacements with norm
\begin{equation}
    \label{eq:lfinity}
    \Vert \uu(t)\Vert_\infty:={\rm esssup}_{\xx\in \Omega}|\uu(\xx,t)|.
\end{equation}
The $L^2(\Omega;\mathbb{R}^d)$ norm is given by
\begin{equation}
    \label{eq:l2}
    \Vert \uu(t)\Vert_2:=\left(\int_\Omega\,|\uu(\xx,t)|^2\,d\xx\right)^{1/2}.
\end{equation}
We remark at the outset that all positive constants that are independent of $\uu$ are denoted either by $C$ or $K$ unless explicitly stated otherwise.

In what follows we denote a ball of radius $R$ centered at an element $\hat\uu$ of \\ $L^\infty(\Omega;\mathbb{R}^d)$ by,
\begin{align}
\label{ballR}
    B(\hat\uu,R) = \left\{ \uu : \norm{\uu  - \hat\uu}_\infty  \le R \right\}.
\end{align}

The derivative of $\LL$ at $\uu$ is denoted by $D(\LL)[\uu]$ and is a linear operator on $L^\infty(\Omega;\mathbb{R}^d)$. 
The operator norm of $D(\LL)[\uu]$ is written $\normiii{D(\LL)[\uu]}$ and defined by
\begin{align}
    \label{operatornorm} 
    \normiii{D(\LL)[\uu]}=\sup_{\Delta\uu\in L^\infty(\Omega;\mathbb{R}^d)}\frac{\Vert D(\LL)[\uu]\Delta\uu\Vert_\infty}{\Vert\Delta\uu \Vert_\infty}.
\end{align}
From Proposition \ref{propn:deriv} of section \ref{sec:existence theory for soft} the derivative is understood as a bounded linear map on $L^\infty(\Omega;\mathbb{R}^d)$.  For this case it is shown that there is a fixed positive constant $C>0$ independent of $\epsilon$ and $\Delta\uu$ in $L^\infty(\Omega;\mathbb{R}^d)$ such that
\begin{equation}\label{bdoperator}
\Vert D(\LL)[\uu]\Delta\uu\Vert_{\infty}\leq \frac{C}{\epsilon^2} \Vert \Delta \uu\Vert_{L^\infty(\Omega;\mathbb{R}^d)},
\end{equation}

The inverse map  when it exists on $\mathcal{V}$ is written $D(\LL)[\hat{\uu}]^{-1}$ and 
\begin{align}
    \label{operatornorminv}
    \Vert D(\LL)[\uu]^{-1}\Delta\uu\Vert_\infty\geq \frac{\epsilon^2}{C}\Vert\Delta\uu\Vert_\infty.
\end{align}
We define 
\begin{align}
    \label{operatornorminv2} 
    \normiii{D(\LL)[\uu]^{-1}}=\sup_{\Delta\uu\in \mathcal{V}}\frac{\Vert D(\LL)[\uu]^{-1}\Delta\uu\Vert_\infty}{\Vert\Delta\uu \Vert_\infty}.
\end{align}
{
In this treatment we consider a very general class of domains $\Omega$.  These are the domains that satisfy the interior cone condition \cite{adams2003sobolev}. }
{
\begin{defn}
\label{D2}
The interior cone condition for $\Omega$ states that there exists a positive constant angle $\theta>0$ such that any $\xx\in\Omega$ contains a spherical cone $C_{\lambda,\theta}(\xx,\ee_{\xx})$ with its apex at $\xx$, radius $\lambda$ aperture angle $2\theta$
bisected by an axis in the direction of a unit vector $\ee_{\xx}$.
\end{defn}
Such domains can not have external cusps, see \Cref{fig:int-cone}. Convex domains as well as non convex domains given by notched specimens satisfy the interior cone condition. 
In this treatment all domains $\Omega$ are assumed to satisfy the interior cone condition. }
\begin{figure}
    \centering
\includesvg[width=0.5\linewidth]{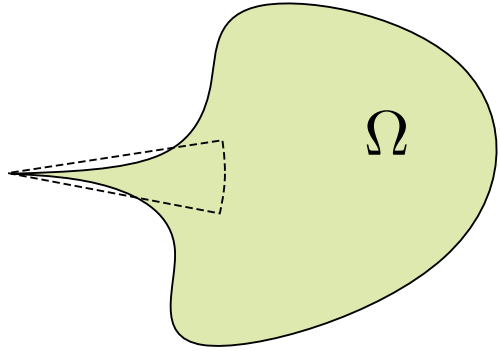}
		  \caption{An example of a domain $\Omega$ that does not satisfy the interior cone condition.}
		  \label{fig:int-cone}
\end{figure}

We now identify the subspace of $L^\infty(\Omega;\mathbb{R}^d)$ in which we find solutions to the load control problems.
First let $\Pi$ denote the space of rigid motions,  i.e.,
\begin{align}
    \label{eq:def-rigid}
    \Pi = \{ \QQ\xx + \cc : \QQ^T = -\QQ \in \R^{d \times d}, \cc \in \R^d  \}.
\end{align}	
Direct use Lemma 2 of \cite{DuGunLehZho}  shows that the rigid rotations comprise the null space of the strain operator:
\begin{propn}\label{kernelstrain}
    $S(\yy,\xx,\uu)=0$ for all $\xx \in \Omega$ and $\yy \in H_\epsilon(\xx)$ if and only if $\uu\in \Pi$.
\end{propn}
From this proposition it follows that $\LL(\ww)=0$ if $\ww\in\Pi$.

Denote the closure of $\mathcal{V}$ in  $L^2(\Omega;\mathbb{R}^d)$ by $\overline{\mathcal{V}}^2$. The function space $\mathcal{V}$ used  for quasistatic evolutions under load control is given by;
\begin{defn}
\label{H}
The space $\mathcal{V}\subset L^\infty(\Omega;\mathbb{R}^d)$ is defined by
\begin{align}\label{vzero}
\mathcal{V}=\{\uu\in L^\infty(\Omega;\mathbb{R}^d);\,\, \int_\Omega\uu\cdot\ww\,d\xx=0, \hbox{ for all $\ww\in\Pi$}\}.
\end{align}
\end{defn}
This space is closed in the $L^\infty(\Omega;\mathbb{R}^d)$ norm and it follows that $\overline{\mathcal{V}}^2\cap \Pi=\{\boldsymbol{0}\}$.

\begin{lemma}\label{luinv}
For $\uu\in \mathcal{V}$ one has $\LL[\uu]\in\mathcal{V}$ and is a bounded operator. 
\end{lemma}
This follows from the straight forward estimate $\Vert \LL(\uu)\Vert_\infty\leq C\Vert\uu\Vert_\infty$ and $\int_\Omega\LL[\uu]\cdot \ww \,d\xx=0$. {The orthogonality  follows from an integration by parts as in the proof of Lemma \ref{thm:symmet} of section \ref{sec:symmbounded} and Proposition \ref{kernelstrain}.} 

The operator $\LL[\uu]$ is continuously Fr\'echet differentiable on $\mathcal{V}$. 
\begin{thm}
    \label{thm:Frechetderiv}
    The linear transform $D(\LL)[\uu]:\mathcal{V}\rightarrow\mathcal{V}$ is the Fr\'echet derivative of $\LL$ and exists with respect to the $L^\infty(\Omega;\mathbb{R}^d)$ norm for all $\uu\in\mathcal{V}$, i.e.,
\begin{align}\label{freshet}
    \lim_{\Vert\Delta\uu\Vert_\infty\rightarrow 0}\frac{\Vert\LL[\uu + \Delta \uu] (\xx) - \LL[\uu](\xx) -D(\LL)[\uu]\Delta\uu \Vert_\infty}{\Vert\Delta\uu\Vert_\infty}=0. 
\end{align}
Moreover it is Lipshitz continuous in $\uu$, i.e., for $\boldsymbol{\delta},\Delta\uu \in\mathcal{V}$ there is a constant $C$ independent of $\boldsymbol{\delta}$ such that
\begin{align}\label{continuous}
    \frac{\Vert D(\LL)[\uu+\boldsymbol{\delta}]\Delta\uu-D(\LL)[\uu]\Delta\uu\Vert_\infty}{\Vert\Delta\uu\Vert_\infty}\leq C\Vert\boldsymbol{\delta}\Vert_\infty. 
\end{align}
\end{thm}
The theorem is proved in section \ref{sec:existence theory for soft}.


We now and assert the existence of solutions to the load control problem. 

\begin{thm}[Load control]
    \label{thm:exitence-of-perturbation}
    Let $\uu_0, \bb_0 \in  \mathcal{V}$ be such that $\LL[\uu_0] = \bb_0$ and assume ${D(\LL)[\uu_0]^{-1}}$ exists and is bounded on on $\mathcal{V}$, set
    \begin{align}\label{rdefined}
    R = \frac{C\epsilon^{\frac{5}{2} }}{\normiii{D(\LL)[\uu_0]^{-1}} },
    \end{align}
    where $C$ is obtained from \eqref{eq:const}. 
    Then for any given load path such that $\bb(t) : [0, T]\to \mathcal{V}$  is continuous and $\bb(t)\in B\left(\bb_0,\frac{R}{2 \normiii{D(\LL)[\uu_0]^{-1}}}\right)$, there exists a unique continuous solution path $\uu(t) \in \mathcal{V}$ lying inside $B(\uu_{0}, R)$ such that
\begin{align}\label{eq:loadcontrol}
    \LL[\uu(t)] = \bb(t), \hbox{  for $t \in [0, T]$.}
\end{align}	
\end{thm}

We now show that $D(\mathcal{L})[\boldsymbol\uu_0]^{-1}$ exists and is bounded on $\mathcal{V}$ for fields $\uu_0$ inside the strength domain. This includes the case $\uu_0=\boldsymbol{0}$.
We now define the stability tensor.
\begin{defn}{Stability tensor}\label{stabilitytensorload}
\begin{align}\label{stabtenload}
    \mathbb{A}[\uu](\xx) = \int\limits_{H_\epsilon(\xx) \cap \Omega}^{} \frac{J^\epsilon(\abs{\yy - \xx}) }{\epsilon^{d+1} \omega_d|\yy-\xx|}  g''\left( \sqrt{\abs{\yy - \xx}} S\left(\yy, \xx, \uu\right)\right) \ee_{\yy - \xx} \otimes \ee_{\yy - \xx}  d\yy.
\end{align}
\end{defn}
The stability tensor is a measurable tensor valued function.
We write $\mathbb{A}[\uu]-\gamma\mathbb{I}>0$ when for all $\xx\in\Omega$ and all $\vv\in\mathbb{R}^d$,  $\mathbb{A}[\uu](\xx)\vv\cdot\vv-\gamma|\vv|^2>0$. 

When the domain satisfies the interior cone condition and one has a deformation $\uu_0$ inside the strength domain we have;
    \begin{thm}[Invertibility]\label{thm:stabinvI}
    Assume $\uu_0$ lies strictly inside the strength domain, i.e., for all $\xx\in \Omega$ and $\yy\in\Omega\cap H_\epsilon(\xx)$ the strain $\sqrt{\xx-\yy}S(\yy,\xx,\uu_0)$ lies inside a closed interval contained inside the interval $(r^e,r^c)$,  and if $\Omega$ satisfies the interior cone condition
    then $Ker\{D(\LL)[\uu]\}=\{\boldsymbol{0}\}$ and there is a $\gamma>0$ for which $\mathbb{A}[\uu_0]-\gamma\mathbb{I}>0$ so
    $D(\mathcal{L})[\uu_0]^{-1}$ is well defined on $\mathcal{V}$ and $\normiii{D(\mathcal{L})[\uu_0]^{-1}}<\infty$.
    \end{thm}
    
{
Now we show existence of a nontrivial initial data $(\uu_0, \bb_0)$.
Pick any $\uu_0 \ne 0$ in the strength domain of the material and set $\bb_0 = \LL[\uu_0]$. Now follows from \Cref{thm:exitence-of-perturbation} and \Cref{thm:stabinvI}
    that for any given load path such that $\bb(t) : [0, T]\to \mathcal{V}$  is continuous and $\bb(t)\in B\left(\bb_0,\frac{R}{2 \normiii{D(\LL)[\uu_0]^{-1}}}\right)$, there exists a unique continuous solution path $\uu(t) \in \mathcal{V}$ lying inside $B(\uu_{0}, R)$ such that
\begin{align}\label{eq:loadcontrol2}
    \LL[\uu(t)] = \bb(t), \hbox{  for $t \in [0, T]$.}
\end{align}	
\\
However even with this initial condition Theorem \ref{thm:exitence-of-perturbation} does not prohibit a solution $\uu(t)$ associated with a propagating crack like defect.
} 
    
We introduce the peridynamic potential energy.
\begin{align}
    \label{eq:energy}
    PD[\uu] =  \int\limits_{\Omega}^{}  \int\limits_{H_\epsilon(\xx) \cap \Omega}^{} \abs{\yy - \xx} \W^\epsilon(S(\yy, \xx, \uu)) d\yy d\xx.
\end{align}	

The total energy of the system is given by
\begin{align}
    \label{eq:total_energy}
    E[\uu] =  PD[\uu] - \int\limits_{\Omega}^{} \uu \cdot \bb_0 \ d\xx .
\end{align}	
The critical point $\uu_0$ for the total energy \eqref{eq:total_energy} is given by the solution to the Euler Lagrange equation
\begin{align}
    \label{eq:first-var}
  D E[\uu_0] = 0
\end{align}	
which is  
\begin{align*}
    \LL[\uu_0] = \bb_0.
\end{align*}

 \begin{thm}
    \label{thm:Unique}
     Assume $\uu_0$  lies strictly inside the strength domain and  is a critical point of the  energy $E[\uu]$ given by \eqref{eq:total_energy} for the choice $\bb = \bb_0$. 
     Then
    \begin{equation}\label{minimloc}
	E[\uu] \geq E[\uu_0],
    \end{equation}	
    over all displacements $\uu$ in the strength domain within the ball $B(\uu_0,R)$.
\end{thm}	
    The Theorem \ref{thm:Unique} is proved in section \ref{sec:energyminimize}.
    
    On combining theorems \ref{thm:exitence-of-perturbation}, \ref{thm:stabinvI}, and \ref{thm:Unique} we discover that a quasistatic PD  evolution exists within a neighborhood of a local energy minimizer over displacements inside the strength the domain.
    \begin{thm}
    \label{thm:Exist-hard loading}
     Assume $\uu_0$  lies strictly inside the strength domain and  is a critical point of the  energy $E[\uu]$ given by \eqref{eq:total_energy} for the choice $\bb = \bb_0$, 
    Under the hypothesis of Theorem \ref{thm:stabinvI} we can choose ${R}$ given by \eqref{rdefined} 
    such that for any continuous load path $\bb(t) : [0, T]\to \mathcal{V}$  with $\bb(t)\in B\left(\bb_0,\frac{R}{2 \normiii{D(\LL)[\uu_0]^{-1}}}\right)$, there exists a unique continuous solution path $\uu(t) \in \mathcal{V}$ lying inside $B(\uu_{0}, R)$ such that
\begin{align}\label{eq:equlib-t}
    \LL[\uu(t)] = \bb(t), \hbox{  for $t \in [0, T]$.}
\end{align}	
Here $\uu_0$ is a critical point and minimizer of the energy $E[\uu]$ for the choice $\bb=\bb_0$, 
    \begin{equation}\label{minimloc2}
	E[\uu] \geq E[\uu_0],
    \end{equation}	
    over all displacements $\uu$ belonging to the strength domain inside the ball $B(\uu_0,R)$.
\end{thm}
The energy inequality \eqref{minimloc2} of Theorem \ref{thm:Exist-hard loading} is proved in section \ref{sec:energyminimize}.

When $\uu_0$ does not lie in the strength domain and $\Omega$ does not satisfy an interior cone condition then one can no longer assert the existence of a $\gamma>0$ such that $\A-\gamma\I>0$. With this in mind
we present a necessary condition for the invertibility of $D(\mathcal{L})[\uu_0]$.


\begin{thm}[Necessity condition for invertibility]
    \label{thm:ness}
    Given $\uu_0\in\mathcal{V}$. If $D(\mathcal{L})[\uu_0]^{-1}$ exists on $\mathcal{V}$ then $Ker\{\mathbb{A}[\uu_0](\xx)\}=\{\boldsymbol{0}\}$ almost everywhere on $\Omega$.
\end{thm}

{This necessity condition is seen in the more general sufficient conditions for invertibility given in Theorem \ref{esistence of inverse in strength domain}. Part of the sufficient conditions require that all eigenvalues of $\mathbb{A}[\uu]$ lie outside an open interval containing $0$.}

{To prove the theorem we suppose $Ker\{\mathbb{A}[\uu_0](\xx)\}\not=\{\boldsymbol{0}\}$ on a set $\mathcal{F}\subset \Omega$ of nonzero Lebesgue measure  and proceed to construct an explicit sequence of nonzero elements $\{\ww_n\}\in\mathcal{V}$ to show that $\Vert D(\LL)[\uu_0]\ww_n\Vert_\infty\rightarrow 0$ but that $\Vert\ww_n\Vert_{\infty}=1$. This proves that $D(\mathcal{L})[\uu_0]^{-1}$ is not defined on $\mathcal{V}$. The proof is given in section \ref{sec:necessary}.}


When the load path $[0,T]\rightarrow\bb(t)$ is differentiable and its derivative is continuous in $t$ then energy balance holds. This is codified in the following theorem.
{\begin{thm}[Energy balance]
    \label{thm:Energy-load}
    If $\uu_0,\bb_0$ satisfy the hypotheses of theorem \ref{thm:exitence-of-perturbation}, $\bb(t)$ is continuously differentiable with respect to $t$ and $D(\LL)[\uu]^{-1}$ exists for all $\uu\in \overline{B(\uu_0,R)}$,
    then the derivative $ \frac{\partial\uu(t,\xx)}{\partial t}$ belongs to $L^\infty(\Omega;\mathbb{R}^3)$ and is continuous in $t$ and is related to the loading rate $\bb_t$ by
  \begin{align}\label{eq:limzero for partial uuthm}
   D(\mathcal{L})[\uu(t)]^{-1}\left(\frac{\partial \bb(t)}{\partial t}\right)
    =\frac{\partial \uu(t)}{\partial t},
\end{align}
 and we have the energy balance for $0<t<T$, given by
   \begin{align*}
	PD[\uu(t)] - \int\limits_{\Omega}^{} \uu(t) \cdot \bb(t) \ d\xx &= PD[\uu_0] - \int\limits_{\Omega}^{} \uu_0 \cdot \bb_0 \ d\xx \\
	-\int_0^t\int_\Omega\uu(\tau)\cdot\bb_t(\tau)\,d\xx\,d\tau,
    \end{align*}
    where $PD(\uu)$ is given by \eqref{eq:energy} and $\bb_t$ is the derivative of $\bb(t)$.
\end{thm}
If instead of having existence of $D(\LL)[\uu]^{-1}$ for all $\uu\in \overline{B(\uu_0,R)}$, we only know that $D(\LL)[\uu_0]^{-1}$ exists we can appeal to Banach's Lemma to find a quasistatic evolution $\uu(t)$ that satisfies the energy balance in the neighborhood of $\uu_0$. This is illustrated in section \ref{energybalance}. }

\section{Existence theory for load control}%
\label{sec:existence theory for soft}

In this section we give the proofs of theorem \ref{thm:Frechetderiv} and theorem \ref{thm:exitence-of-perturbation}. We first prove theorem \ref{thm:exitence-of-perturbation} as theorem \ref{thm:Frechetderiv} follows from techniques developed in its proof. In order to prove theorem \ref{thm:exitence-of-perturbation} we establish the necessary prerequisites.
For a fixed $\uu\in\mathcal{V}$ we define the linear transform $D(\LL)[\uu]$ acting on $\Delta\uu\in\mathcal{V}$ by
\begin{align}
    \label{eq:derivative}
    &D(\LL)[\uu] \Delta \uu =\nonumber\\
    &-\int\limits_{H_\epsilon(\xx) \cap \Omega}^{} \frac{J^\epsilon(\abs{\yy - \xx}) }{\epsilon^{d+1} \omega_d}  g''\left( \sqrt{\abs{\yy - \xx}} S\left(\yy, \xx, \uu\right)\right) S(\yy, \xx, \Delta \uu) \ee_{\yy - \xx} d\yy,
\end{align}	
and we first note that it  a linear functional on $L^\infty(\Omega;\mathbb{R}^2)$.
\begin{propn}
   \label{propn:deriv}
    The linear functional $D(\LL)[\uu]$ exists for all $\uu\in \cal{V}$ and is a bounded linear functional for  $\Delta \uu\in L^\infty(\Omega;\mathbb{R}^d)$. 
\end{propn}
\proof{}
Recalling $g''(0)=\max\{|g''(r)|\}$ the proposition follows from the string of inequalities
\begin{align*}
    &\Vert D(-\LL)[\uu] \Delta \uu \Vert_\infty \leq\\
   & {\rm esssup}_{\xx\in \Omega}\{ \int\limits_{H_\epsilon(\xx) \cap \Omega}^{} |\frac{J^\epsilon(\abs{\yy - \xx}) }{\epsilon^{d+1} \omega_d}  g''\left( \sqrt{\abs{\yy - \xx}} S\left(\yy, \xx, \uu\right)\right) S(\yy, \xx, \Delta \uu) \ee_{\yy - \xx}| d\yy \}\\
   &\leq C\,{\rm esssup}_{\xx\in \Omega}\{\frac{1}{\epsilon^{d+1}\omega_d} \int\limits_{H_\epsilon(\xx) \cap \Omega}^{} \frac{|\Delta \uu| }{|\yy-\xx|} d\yy \}
   \leq \frac{C}{\epsilon^{2}}||\Delta \uu||_\infty,
\end{align*}
where the constant $C$ is independent of $\Delta\uu$ and $\epsilon.$
\qed

The existence of a quasistatic evolution is based on fixed point theory. 
We start by proposing a functional $ T_{\bb - \bb_0}(\uu - \uu_0)$ defined on $B(\uu_0,R)$ such that at its fixed point $\uu^\ast-\uu_0$ one has that $\LL[\uu^\ast] = \bb$.
\begin{propn}
\label{propn:FixedPoint}
    Define
\begin{align}
\label{eq:hardInner}
    F_{\uu_0}(\uu - \uu_0) =& -\LL[(\uu-\uu_0) +\uu_0] + \LL[\uu_0] - D(-\LL)[\uu_0] (\uu - \uu_0)\\
    =& -\LL[\uu] + \LL[\uu_0] - D(-\LL)[\uu_0] (\uu - \uu_0)\nonumber
    \end{align}	
and the map
\begin{align}
\label{eq:fixed point map}
    T_{\bb - \bb_0}(\uu - \uu_0) 
    & = -D(-\LL)[\uu_0]^{-1} \left(F_{\uu_0}(\uu - \uu_0) +  (\bb - \bb_0)\right).
\end{align}	
If 
$
    T_{\bb - \bb_0}(\uu^* - \uu_0) = \uu^* - \uu_0,
    $
then $\LL[\uu^*] = \bb$. 
\end{propn}
\proof
Applying $-D(-\LL)[\uu_0]$ to the fixed point equation
\begin{align*}
    T_{\bb - \bb_0}(\uu^* - \uu_0) = \uu^* - \uu_0.
\end{align*}	
we get from the definition of $T, F$ and $\LL[\uu_0] = \bb_0$ that
\begin{align*}
    F_{\uu_0}(\uu^* - \uu_0) + (\bb - \bb_0) &= -D(-\LL)[\uu_0](\uu^* - \uu_0) \\
    -\LL[\uu^*] + \LL[\uu_0] - D(-\LL)[\uu_0] (\uu^* - \uu_0) 
     + (\bb - \bb_0) &= -D(-\LL)[\uu_0](\uu^* - \uu_0) \\
    -\LL[\uu^*] + \LL[\uu_0]  
     + (\bb - \bb_0) &= 0 \\
     \LL[\uu^*] &= \bb.
\end{align*}	
Therefore $\uu^*$ solves equation \eqref{eq:pdQuasi} with body force density $\bb$.
\qed

The  proof of Theorem \ref{thm:exitence-of-perturbation} now proceeds in two steps. Step I shows
that $T_{\bb - \bb_0}$ is a contraction on the ball on $B(\uu_0, R)$, defined by \eqref{ballR} and therefore has a unique fixed point. Step II shows that for any prescribed continuous load path $\bb(0)=\bb_0$ and $\bb(t)\in \mathcal{V}$, such that $\bb: [0, T] \to  B\left(\bb_0,\frac{R}{2 \norm{D^{-1}(-\LL)[\uu_0]}_\infty}\right)$there exists a unique continuous solution path $\uu(t) \in \mathcal{V}$ lying inside $B(\uu_0, R)$.

\proof[Proof of Theorem \ref{thm:exitence-of-perturbation}]

{\bf Step I}. 
Choose $\uu$, $\uu'$ in $B(\uu_0,R)$ and from the definition of $T_{\bb - \bb_0}$,
\begin{align*}
   & \norm{T_{\bb - \bb_0} (\uu - \uu_0) - T_{\bb - \bb_0}(\uu' - \uu_0)}_{\infty}\\
    & = \norm{D(-\LL)[\uu_0]^{-1} \left( F_{\uu_0}(\uu - \uu_0) - F_{\uu_0}(\uu' - \uu_0) \right)}_{\infty}.
\end{align*}	

Using the linearity of $D(-\LL)[\uu_0]$ gives
\begin{align*}
    F_{\uu_0}(\uu - \uu_0) - F_{\uu_0}(\uu' - \uu_0)
    & = -\LL(\uu) + \LL(\uu') - D(-\LL)[\uu_0] (\uu - \uu').
\end{align*}	
Recall from the definition of $\LL$, 
\begin{align*}
    & -\LL(\uu) + \LL(\uu') \\
    & = 
    \int\limits_{H_\epsilon(\xx) \cap \Omega}^{} \frac{J^\epsilon(\abs{\yy - \xx} )}{\epsilon^{d+1} \omega_d} \frac{g'(\sqrt{\abs{\yy - \xx}} S(\yy, \xx, \uu)) - g'(\sqrt{\abs{\yy - \xx} } S(\yy, \xx, \uu' )} {\sqrt{\abs{\yy - \xx} } } \ee_{\yy - \xx} d\yy \\
    & = 
    \int\limits_{H_\epsilon(\xx) \cap \Omega}^{} \frac{J^\epsilon(\abs{\yy - \xx} )}{\epsilon^{d+1} \omega_d} \int\limits_{0}^{1} g''(r(t)) dt\ S(\yy, \xx, \uu - \uu') \ee_{\yy - \xx} d\yy,
\end{align*}	
where we have used
\begin{align*}
     g'(\sqrt{\abs{\yy - \xx}} S(\yy, \xx, \uu)) - g'(\sqrt{\abs{\yy - \xx} } S(\yy, \xx, \uu' ) 
     & = 
    \int\limits_{\sqrt{\abs{\yy - \xx}} S(\yy, \xx, \uu')}^{\sqrt{\abs{\yy - \xx} } S(\yy, \xx, \uu )} g''(r) dr  
    \\
     & = 
    \int\limits_{0}^{1} g''(r(t)) r'(t) dt
\end{align*}	
with
\begin{align*}
    r(t) =\sqrt{\abs{\yy - \xx} }  ((1 - t) S(\yy, \xx, \uu') + t S(\yy, \xx, \uu))
\end{align*}	
and hence 
\begin{align*}
    r'(t) = \sqrt{\abs{\yy - \xx} } (S(\yy, \xx, \uu) - S(\yy, \xx, \uu')) = \sqrt{\abs{\yy - \xx} }  S(\yy, \xx, \uu - \uu').
\end{align*}	
Also recall from equation \eqref{eq:derivative}
\begin{align*}
&    D(-\LL)[\uu_0] (\uu - \uu') = \\
&     \int\limits_{H_\epsilon(\xx) \cap \Omega}^{} \frac{J^\epsilon(\abs{\yy - \xx)} }{\epsilon^{d+1} \omega_d}  g''\left( \sqrt{\abs{\yy - \xx}} S\left(\yy, \xx, \uu_0\right)\right) S(\yy, \xx,  \uu - \uu') \ee_{\yy - \xx} d\yy.
\end{align*}	
Therefore,
\begin{align*}
    &F_{\uu_0}(\uu - \uu_0) - F_{\uu_0}(\uu' - \uu_0) 
      = -\LL(\uu) + \LL(\uu') - D(-\LL)[ \uu_0]( \uu - \uu') \\
    & = \int\limits_{H_\epsilon(\xx) \cap \Omega}^{} \frac{J^\epsilon(\abs{\yy - \xx} )}{\epsilon^{d+1} \omega_d}  \int\limits_{0}^{1} \left[ g''  \left( \sqrt{\abs{\yy  - \xx} } S\left(\yy, \xx, (1-t)\uu' + t \uu \right) \right) - g''\left(\sqrt{\abs{\yy -\xx} } S(\yy, \xx, \uu_0) \right) \right]  dt\\
    &\quad \quad \quad
     S(\yy, \xx, \uu - \uu') \ee_{\yy - \xx} d\yy \\
    & = \int\limits_{H_\epsilon(\xx) \cap \Omega}^{} \frac{J^\epsilon(\abs{\yy - \xx} )}{\epsilon^{d+1} \omega_d}  \int\limits_{0}^{1} \int\limits_{0}^{1} g'''(s_t(\tau)) d\tau \sqrt{\abs{\yy - \xx} } S(\yy, \xx, (1-t) \uu' + t\uu - \uu_0)  dt\ 
    \\&\quad \quad \quad
     S(\yy, \xx, \uu - \uu') \ee_{\yy - \xx} d\yy,
\end{align*}	
where again, we have used
\begin{align*}
    g''(r(t)) - g''(\sqrt{\abs{\yy - \xx} S(\yy, \xx, \uu_0)}  = \int\limits_{0}^{1} g'''(s_t(\tau)) s_t'(\tau) d\tau
\end{align*}	
with
$
    s_t(\tau) = \sqrt{\abs{\yy - \xx} }  ((1 - \tau) S(\yy, \xx, \uu_0) + \tau\ S(\yy, \xx, (1-t) \uu' + t \uu))
    $
and hence
\begin{align*}
    s_t'(\tau) & = \sqrt{\abs{\yy - \xx} }  S (\yy, \xx, (1-t)\uu' + t \uu - \uu_0).
\end{align*}	
So, 
\begin{align*}
    & F_{\uu_0}(\uu - \uu_0) - F_{\uu_0}(\uu' - \uu_0)  \\
    & = \int\limits_{H_\epsilon(\xx) \cap \Omega}^{} \frac{J^\epsilon(\abs{\yy - \xx} )}{\epsilon^{d+1} \omega_d}  \int\limits_{0}^{1} \int\limits_{0}^{1} g'''(s_t(\tau)) d\tau \sqrt{\abs{\yy - \xx} } S(\yy, \xx, (1-t) \uu' + t\uu - \uu_0)  dt\ 
    \\
    &\quad\quad\quad S(\yy, \xx, \uu - \uu') \ee_{\yy - \xx} d\yy
\end{align*}	
We now work to bound the integrand. First using the Cauchy-Schwartz inequality and $\abs{\ee_{\yy - \xx}}  = 1$, to see that
$
    \abs{S(\yy, \xx, \uu)}  \le \frac{\abs{\uu(\yy) - \uu(\xx)}}{\abs{\yy - \xx} }.
    $
Now define 
\begin{align*}
    \uu_t := (1 - t) \uu' + t \uu,
\end{align*}	
and
\begin{align*}
    & \abs{
    F_{\uu_0}(\uu - \uu_0) - F_{\uu_0}(\uu' - \uu_0) 
} \\
    & \le
    \frac{\norm{J^\epsilon}_\infty }{\epsilon^{d+1} \omega_d} \norm{g'''}_{\infty}  \int\limits_{H_\epsilon(\xx) \cap D}^{} \frac{\sqrt{\abs{\yy - \xx} } }{\abs{\yy - \xx}^2}
    \int\limits_{0}^{1} 
    \\
    &\quad\quad\quad \abs{(\uu_t - \uu_0)(\yy) - (\uu_t - \uu_0)(\xx)} dt \abs{(\uu - \uu')(\yy) - (\uu - \uu')(\xx)} d\yy 
    \\
    & \le
    \frac{4 \norm{J^\epsilon}_\infty }{\epsilon^{d+1} w_d} \norm{g'''}_{\infty}
    \norm{\uu - \uu'}_{\infty}
    \int\limits_{H_\epsilon(\xx) \cap \Omega}^{} \frac{1}{\abs{\yy - \xx}^{\frac{3}{2} }} d\yy \int\limits_{0}^{1}  \norm{\uu_t - \uu_0}_{\infty} dt,
\end{align*}	
where
\begin{align*}
    \int\limits_{H_\epsilon(\xx)}^{} \abs{\yy - \xx}^{-\frac{3}{2} } d\yy
    \le \omega_d \int\limits_{0}^{\epsilon} \rho^{-\frac{3}{2} } \rho^{d-1} d\rho
    = \frac{\omega_d}{d- \frac{3}{2} }   \epsilon^{d - \frac{3}{2} },
\end{align*}	
so for all $\xx \in D$,
\begin{align}
    \label{eq:F-diff}
     \abs{
    F_{\uu_0}(\uu - \uu_0) - F_{\uu_0}(\uu' - \uu_0) 
} 
    & \le C(\epsilon, g, J) \norm{\uu - \uu'}_{\infty} \int\limits_{0}^{1} \norm{\uu_t - \uu_0}_{\infty}dt,
\end{align}	
where
\begin{align}
\label{eq:const}
    C(\epsilon, g, J) = \frac{ \norm{J^\epsilon}_\infty }{\epsilon^{\frac{5}{2} }} \norm{g'''}_{\infty}.
\end{align}	
We have
\begin{align*}
    \abs{\uu_t - \uu_0} = \abs{(1-t) \uu' + t \uu - (1-t) \uu_0 - t \uu_0}
    & = \abs{(1-t) (\uu' - \uu_0) + t (\uu - \uu_0)},
\end{align*}	
therefore, if $\uu, \uu' \in B(R, \uu_0)$, 
\begin{align}\label{taylorsremainder}
    \norm{\uu_t - \uu_0}_{\infty} \le (1-t) \norm{\uu' - \uu_0}_{\infty} + t \norm{\uu - \uu_0}_{\infty} 
    \le (1-t) R + t R = R
\end{align}	
so
\begin{align*}
\int\limits_{0}^{1}  \norm{\uu_t - \uu_0}_{\infty} dt \le R.
\end{align*}	
It now follows from equation \eqref{eq:F-diff} that
\begin{align}
    \label{eq:inf-bound-diff}
     \norm{
    F_{\uu_0}(\uu - \uu_0) - F_{\uu_0}(\uu' - \uu_0) 
}_{\infty}
    & \le C(\epsilon, g, J) \norm{\uu - \uu'}_{\infty} R,
\end{align}	
and
\begin{align*}
    &\norm{T_{\bb - \bb_0} (\uu - \uu_0) - T_{\bb - \bb_0}(\uu' - \uu_0)}_{\infty} \\
    & \le \norm{D(-\LL)[\uu_0]^{-1} \left( F_{\uu_0}(\uu - \uu_0) - F_{\uu_0}(\uu' - \uu_0) \right)}_{\infty} \\
    & \le \norm{D(-\LL)[\uu_0]^{-1}}_{\infty} \norm{\left( F_{\uu_0}(\uu - \uu_0) - F_{\uu_0}(\uu' - \uu_0) \right)}_{\infty} \\
    & \le R \ C(\epsilon, g, J) \normiii{D(-\LL)[\uu_0]^{-1}}
     \norm{\uu - \uu'}_{\infty} 
\end{align*}	
Last choosing $R$ such that
\begin{align}
    \label{eq:choice-R}
     C(\epsilon, g, J) \normiii{D(-\LL)[\uu_0]^{-1}}  R = \half
\end{align}	
it follows that $T_{\bb - \bb_0}$ is a contraction for $\uu$ and $\uu'$ in the ball $B(\uu_0, R)$. 





It remains to show that  if $\norm{\uu - \uu_0}_{\infty} \le R$ then $\norm{T_{\bb - \bb_0}(\uu - \uu_0)}_{\infty} \le R$.
Indeed, for $\uu \in B(\uu_0, R)$,
\begin{align*}
    &\norm{T_{\bb - \bb_0}(\uu - \uu_0)}_{\infty}\leq\\
    &\le \norm{D(-\LL)[\uu_0]^{-1}( F_{\uu_0}(\uu - \uu_0))}_{\infty} + \norm{D(-\LL)[\uu_0]^{-1}(\bb - \bb_0)}_{\infty} \\
    & \le \norm{D(-\LL)[\uu_0]^{-1}}_{\infty} \norm{F_{\uu_0}(\uu - \uu_0)}_{\infty}  + 
    \norm{D(-\LL)[\uu_0]^{-1}}_{\infty} \norm{\bb - \bb_0}_{\infty}.
\end{align*}	

Taking $\uu' = \uu_0$ in \eqref{eq:F-diff} gives 
\begin{align*}
     \norm{F_{\uu_0}(\uu - \uu_0)}_{\infty}
    & = \norm{F_{\uu_0}(\uu - \uu_0) - F_{\uu_0}(\uu_0 - \uu_0)  }_{\infty}  \\
    & \le C(\epsilon, g, J) \norm{\uu - \uu_0}_{\infty} \int\limits_{0}^{1} t \norm{\uu - \uu_0} _{\infty} dt  \\
    & = C(\epsilon, g, J) \norm{\uu - \uu_0}^2_{\infty}.
\end{align*}	
Then choose $\bb$ with
\begin{align*}
    \norm{\bb - \bb_0}_{\infty} \le \frac{R}{2 \norm{D^{-1}(-\LL)[\uu_0]} }_{\infty}.
\end{align*}	
to get
\begin{align*}
    \norm{T_{\bb - \bb_0}(\uu - \uu_0)}_{\infty} 
    & \le C(\epsilon, g, J) \norm{D(-\LL)[\uu_0]^{-1}}_{\infty} R^2  + 
    \frac{R}{2} .
\end{align*}	

The choice of $R$ from equation \eqref{eq:choice-R} implies
\begin{align*}
    \norm{T_{\bb - \bb_0}(\uu - \uu_0)}_{\infty}  \le R.
\end{align*}	
So for this choice of $R$, Banach's fixed point theorem implies that there exists unique fixed point $\uu^* \in B(\uu_0, R)$  such that
\begin{align*}
    T_{\bb - \bb_0}(\uu^* - \uu_0) = \uu^* - \uu_0.
\end{align*}	
Step I now follows from Proposition \ref{propn:FixedPoint}.

    


{\bf Step II}.
To show the continuity of $T$ in $\bb - \bb_0$, note
\begin{align}\label{Tcont}
    &\sup_{\uu - \uu_0 \in B(0,R)} \norm{T_{\bb - \bb_0}(\uu - \uu_0)
    - T_{\widetilde{\bb} - \bb_0}(\uu - \uu_0)}_{\infty} \\
    & = \norm{D(-\LL)[\uu_0]^{-1}(\bb - \widetilde{\bb})}_\infty  
    \\
    & \le \norm{D(-\LL)[\uu_0]^{-1}}_{\infty} \norm{\bb - \widetilde{\bb}}_\infty.\nonumber
\end{align}	
Now suppose for $\bb-\bb_0$ we have the fixed point $T_{\bb - \bb_0}(\uu - \uu_0) =\uu-\uu_0$ and for $\widetilde{\bb}-\bb_0$ we have the fixed point $T_{\widetilde{\bb} - \bb_0}(\widetilde{\uu} - \uu_0) =\widetilde{\uu}-\uu_0$. So from \eqref{Tcont} and since $T_{\bb-\bb_0}$ is a contraction we get
\begin{align}\label{Tcont2}
    \Vert\uu - \widetilde{\uu}\Vert_\infty & =\norm{T_{\bb - \bb_0}(\uu - \uu_0) - T_{\widetilde{\bb} - \bb_0}(\widetilde{\uu} - \uu_0)}_{\infty} \\
    & = \norm{T_{\bb - \bb_0}(\uu - \uu_0) - T_{{\bb} - \bb_0}(\widetilde{\uu} - \uu_0)}_{\infty} \nonumber\\
    & + \norm{T_{\bb - \bb_0}(\widetilde{\uu} - \uu_0) - T_{\widetilde{\bb} - \bb_0}(\widetilde{\uu} - \uu_0)}_{\infty}\nonumber\\
    & \leq\frac{1}{2}\Vert\uu-\widetilde{\uu}\Vert_\infty+\Vert D(-\mathcal{L})[\uu_0]^{-1}\Vert_\infty\Vert\bb-\widetilde{\bb}\Vert_\infty \nonumber.
\end{align}	
Hence we have the continuity given by
\begin{align}\label{Tcont3}
    \Vert\uu - \widetilde{\uu}\Vert_\infty \leq 2 \Vert D(-\mathcal{L})[\uu_0]^{-1}\Vert_\infty\Vert\bb-\widetilde{\bb}\Vert_\infty.
\end{align}	
Let $\bb(t)$ be a continuous map 
$\bb: [0, T] \to B\left(\bb_0, {\frac{R}{2 \norm{D(-\LL)[\uu_0]^{-1}}_\infty}}\right)
$ with $\bb(0) = \bb_0$.
Theorem \ref{thm:exitence-of-perturbation} implies that for any $t \in (0, T]$ the solution $\uu(t)$ to equation \eqref{eq:pdQuasi} belongs to the ball $B(\uu_0, R))$. Step II now follows from the continuity \eqref{Tcont3} of the map $T_{\bb(t) - \bb_0}$ on $t$.
Theorem \ref{thm:exitence-of-perturbation} now follows from steps I and II.
\qed

The proof of Theorem \ref{thm:Frechetderiv} now follows from observations made in the proof of Theorem \ref{thm:exitence-of-perturbation}. 

\proof[Proof of Theorem \ref{thm:Frechetderiv}]
Note that for any $\uu'$, $\uu$ belonging to $\mathcal{V}$ the substitutions $\uu'=\uu_0$ and $\uu-\uu'=\Delta\uu$  in  \eqref{eq:F-diff}   give
\begin{align}
    \label{eq:F-diffR}
    \begin{split}
     \Vert -\LL(\uu_0+\Delta\uu) + \LL(\uu_0) - & D(-\LL)[ \uu_0] \Delta\uu\Vert_\infty\\
    & \leq C(\epsilon, g, J) \norm{\Delta\uu}_{\infty} \int\limits_{0}^{1} \norm{\uu_t - \uu_0}_{\infty}dt.
    \end{split}
\end{align}	
Equation \eqref{freshet} of the Theorem \ref{thm:Frechetderiv}  becomes evident noting that  $0\leq\Vert \uu_t - \uu_0 \Vert_\infty\leq\Vert \Delta\uu \Vert_\infty$. Equation \eqref{continuous} follows on writing
\begin{align*}
    &\Vert D(\LL)[\uu+\boldsymbol{\delta}]\Delta\uu-D(\LL)[\uu]\Delta\uu\Vert_\infty\nonumber\\
    & = \int\limits_{H_\epsilon(\xx) \cap \Omega}^{} \frac{J^\epsilon(\abs{\yy - \xx} )}{\epsilon^{d+1} w_d}  \left[ g''  \left( \sqrt{\abs{\yy  - \xx} } S\left(\yy, \xx,  \uu+\boldsymbol{\delta} \right) \right) - g''\left(\sqrt{\abs{\yy -\xx} } S(\yy, \xx, \uu) \right) \right]  \\
    &\quad\quad\quad S(\yy, \xx, \Delta\uu ) \ee_{\yy - \xx} d\yy  \\
    & = \int\limits_{H_\epsilon(\xx) \cap \Omega}^{} \frac{J^\epsilon(\abs{\yy - \xx} )}{\epsilon^{d+1} w_d}  \int\limits_{0}^{1} g'''(\sqrt{|\yy-\xx|}S(\yy,\xx,\uu+t\boldsymbol{\delta}) \sqrt{\abs{\yy - \xx} } S(\yy, \xx,\boldsymbol{\delta})  dt\ 
    \\ & \quad\quad\quad
     S(\yy, \xx, \Delta\uu) \ee_{\yy - \xx} d\yy
\end{align*}
and estimating as in \eqref{eq:F-diff}.
\qed

\section{Stability and Invertibility}%
\label{sec:invertibilityI}
In this section we give the proof of Theorem \ref{thm:stabinvI}. 
We begin with useful properties and necessary observations. Denote the characteristic function $\chi$ of the set $\Omega$ and $\chi_\epsilon$  for the set given by the ball of radius $\epsilon$ with center $0$ denoted by $H_{\epsilon}(0)$. Here the characteristic functions take the value $1$ for points inside the set and zero outside. Set $\chi(\yy,\xx)=\chi(\yy)\chi(\xx)\chi_{\epsilon}(|\yy-\xx|)$ and for $\uu\in\overline{\mathcal{V}}^2$ write
\begin{align}\label{rho}
    \rho(\yy,\xx,S(\yy,\xx,\uu)) = \frac{\chi(\yy,\xx)J^\epsilon(\abs{\yy - \xx})}{\epsilon^{d+1} \omega_d\abs{\yy - \xx}} g''(\sqrt{|\yy-\xx|}S(\yy,\xx,\uu))),
\end{align}
$d=2,3$. Here $\rho(\yy,\xx,S(\yy,\xx,\uu))$ changes sign depending on the factor $g''(r)$ with $r=S(\yy,\xx,\uu)$, see for example figure \ref{ConvexConcavea}. One has the estimate
\begin{align}\label{rhoest}
    |\rho(\yy,\xx,S(\yy,\xx,\uu))|\leq \frac{C\chi(\yy,\xx)}{\epsilon^{d+1} \omega_d\abs{\yy - \xx}}, \qquad C=g''(0)\max_{\yy\in H_1(0)}\{J(|\yy|)\}.
\end{align}
where $max_{r\in \mathbb{R}} \{g''(r)\}$ = $g''(0)$.
One readily verifies the interchange symmetry 
\begin{align}
\label{eq:interchange}
\rho(\yy,\xx,S(\yy,\xx,\uu))=\rho(\xx,\yy,S(\xx,\yy,\uu))
\end{align}
 and  
 \begin{align}
\label{eq:positivity}
0<\rho(\yy,\xx,S(\yy,\xx,\uu))
\end{align}
for strains $S(\yy,\xx,\uu)$ inside the strength domain (see Definition \ref{strength defn}). On the boundary of the strength domain we have
$$\rho(\yy,\xx,S(\yy,\xx,\uu))=0 \hbox{ when } S(\yy,\xx,\uu)={r^c}/{\sqrt{|\yy-\xx|}}$$ and $$\rho(\yy,\xx,S(\yy,\xx,\uu))=0 \hbox{ when }S(\yy,\xx,\uu)=-{r^e}/{\sqrt{|\yy-\xx|}}.$$
Combining equation \eqref{eq:derivative} and  \eqref{rho} we express $ D(\LL)[\uu]  \ww $ as
\begin{align}\label{rhoD2}
    D(\LL)[\uu]  \ww = -\int\limits_{ \Omega}^{} \rho(\yy,\xx,S(\yy,\xx,\uu))\left(\ww(\yy) - \ww(\xx) \cdot \ee_{\yy - \xx}\right)  \ee_{\yy - \xx}  d\yy.
\end{align}

We have that 
\begin{lemma}
\label{thm:symmetricv}
Given $\uu\in\mathcal{V}$, $D(\mathcal{L})[\uu]$ is a symmetric bounded operator on $L^2(\Omega;\mathbb{R}^d)$.
\end{lemma}
The proof of this theorem is given in section \ref{sec:symmbounded}.

The operator, $D(\mathcal{L})[\uu](\ww):\overline{\mathcal{V}}^2\rightarrow \overline{\mathcal{V}}^2$  can be split into two parts and is written
\begin{equation}
    \label{decomp}
    D(\mathcal{L})[\uu](\ww)=\mathbb{K}(\ww)+\A(\ww).
\end{equation}
Here 
\begin{equation}
\label{operTORA}
\mathbb{A}(\ww)=\mathbb{A}(\xx)\ww(\xx)
\end{equation}
where $\mathbb{A}(\xx)$,   is the stability tensor given by Definition \ref{stabilitytensorload}. Hence the operator $\mathbb{A}$ is symmetric on $L^2(\Omega;\mathbb{R}^d)$. It is also bounded, this follows from \eqref{rhoest} and the change of variable $\yy-\xx=\xi$. The operator $\mathbb{K}$ is given by 
\begin{equation}
    \label{decompK}
    \mathbb{K}(\ww)=-\sum_{j=1}^d\int_\Omega\,K_{ij}(\xx,\yy)\ww_j(\yy)\,d\yy,
\end{equation}
with kernel $K_{ij}(\xx,\yy)$ in $L^2(\Omega\times\Omega;{Sym}^{d\times d})$,
\begin{equation}
    \label{decompkK}
    K_{ij}(\xx,\yy)=\rho(\yy,\xx,S(\yy,\xx,\uu))\ee_i\ee_j,
\end{equation}
where ${{Sym}}^{d\times d}$ is the space of $d\times d$ symmetric matrices. 
From Theorem \ref{thm:symmetricv} and the symmetry and boundedness of $\mathbb{A}$ it is immediate that $\mathbb{K}$ is bounded  and symmetric on $L^2(\Omega;\mathbb{R}^d)$. Additionally since $\rho(\yy,\xx,S(\yy,\xx,\uu)$  is in $L^2(\Omega\times\Omega;{Sym}^{d\times d})$ one sees that $\mathbb{K}$ is expressed as a limit of rank one operators so $\mathbb{K}$ is compact on $L^2(\Omega;\mathbb{R}^d)$ see, e.g, \cite{Folland}.
The decomposition given by \eqref{decomp} will be used in the proof Theorem \ref{thm:stabinvI} given in Section \ref{subsec:invertibilityIstep1} below.

The proof of Theorem \ref{thm:stabinvI} proceeds in two steps. The first step assumes the map $D(\LL)[\uu]$ satisfies the following assumptions: 
i) there exists a $\gamma>0$ for which\\
$\mathbb{A}[\uu]-\gamma\mathbb{I}>0$, and ii) {$D(\LL)[\uu]$ is a symmetric bounded linear map on $\overline{\mathcal{V}}^2 \subset  L^2(\Omega;\mathbb{R}^d)$}, { with} $Ker\{D(\LL)[\uu]\}=\{\boldsymbol{0}\}$. With these assumptions it is shown that $D(\LL)[\uu]$ maps $\mathcal{V}$ onto itself so $D(\LL)[\uu]^{-1}$ exists. Additionally we use the hypotheses to show that the inverse is  bounded. These hypotheses are special cases of more general hypotheses stated in in theorem \ref{esistence of inverse in strength domain} and established in following subsection. The proof is completed in the second step where it is shown that these assumptions are satisfied when $\uu$ is strictly contained in the strength domain.
The first step is given in section \ref{subsec:invertibilityIstep1} and the second step is given in section \ref{sec:symmbounded}.

\subsection{Step 1 of proof of Theorem \ref{thm:stabinvI}}%
\label{subsec:invertibilityIstep1}
In this section we establish the following theorem:
\begin{thm}\label{esistence of inverse in strength domain}
{If $D(\LL)[\uu]$ is a symmetric bounded linear map on $\overline{\mathcal{V}}^2 \subset  L^2(\Omega;\mathbb{R}^d)$},  { with} $Ker\{D(\LL)[\uu]\}=\{\boldsymbol{0}\}$ and there is a ${\gamma\not=0}$ such that {$\mathbb{A}^2[\uu]\geq \gamma^2\mathbb{I}$}, then $D(\mathcal{L})[\uu]^{-1}$ exists as a bounded linear map on $\mathcal{V}$ with with respect to the $L^\infty(\Omega;\mathbb{R}^d)$ norm.
\end{thm}
The condition $\mathbb{A}^2[\uu]-\gamma^2\mathbb{I}>0$ is equivalent to saying that all eigenvalues of $\mathbb{A}[\uu]$ lie outside an interval about $0$. This hypothesis includes the case $\mathbb{A}[\uu]-\gamma\mathbb{I}>0$ for $\gamma>0$ used in the proof of Theorem \ref{thm:stabinvI}.

In what follows we denote the range of the map $D(\LL)[\uu]$ on $\overline{\mathcal{V}}^2$  by $D(\LL)[\uu](\overline{\mathcal{V}}^2)$.

The first step in accomplishing the goal of this section is to establish the inequality  given by
\begin{lemma}\label{8}
Given $\uu\in\mathcal{V}$ if $D(\LL)[\uu]$ is a bounded, symmetric linear map on ${\overline{\mathcal{V}}^2 \subset } L^2(\Omega;\mathbb{R}^d)$, { with} $Ker\{D(\LL)[\uu]\}=\{\boldsymbol{0}\}$ and there exists ${\gamma\not=0}$ such that {$\mathbb{A}^2[\uu]\geq \gamma^2\mathbb{I}$} 
then there exists a positive constant $K$ independent of $\ww\in \overline{\mathcal{V}}^2\setminus \{\boldsymbol{0}\}$ such that
    \begin{align} \label{inequality}
        \Vert D(\mathcal{L})[\uu](\ww)\Vert_2\geq K\Vert\ww\Vert_2,
    \end{align}
for all $\ww\in \overline{\mathcal{V}}^2\setminus \{\boldsymbol{0}\}$, hence $D(\LL)[\uu]^{-1}$ exists as an $L^2(\Omega;\mathbb{R}^d)$ bounded map on $\overline{\mathcal{V}}^2$.
\end{lemma}
\begin{proof}
First denote the operator norm of $D(\mathcal{L})[\uu]$ in $L^2(\Omega;\mathbb{R}^d)$ by $\normiii{D(\mathcal{L})[\uu]}_2$ and  $\normiii{D(\mathcal{L})[\uu]}_2<\infty$ from hypothesis. By Theorem \ref{thm:symmetricv} $D(\mathcal{L})[\uu]$ is  symmetric on $\overline{\mathcal{V}}^2$ we have
\begin{align}\label{symm}
       \overline{\mathcal{V}}^2= \overline{D(\LL)[\uu](\overline{\mathcal{V}}^2)}^2\oplus Ker\{D(\mathcal{L})[\uu]\},
    \end{align}
    here we have denoted the $L^2(\Omega;\mathbb{R}^d)$ closure of a set $S$ by $\overline{S}^2$.
We now prove the theorem by contradiction. Suppose there exists a sequence $K_j>0$ such that $\lim_{j\rightarrow 0}K_j= 0$, $\ww_{K_j}\in D(\LL)(\overline{\mathcal{V}}^2)$, $\Vert\ww_{K_j}\Vert_2=1$ and 
\begin{align*}
        \Vert D(\mathcal{L})[\uu](\ww_{K_j})\Vert_2\leq K_j.
    \end{align*}

Since $\Vert\ww_{k_j}\Vert_2=1$ there is a subsequence also denoted by $\{\ww_{K_j}\}$ that is weakly converging to $\ww$ in $\overline{\mathcal{V}}^2$. Now we show that $\ww=\boldsymbol{0}$. Denote the $L^2(\Omega;\mathbb{R}^d)$ inner product by $(\cdot,\cdot)$.
Since $D(\mathcal{L})[\uu]$ is a symmetric linear operator on $\overline{\mathcal{V}}^2$ one has for any $\Phi\in\overline{\mathcal{V}}^2$ 
\begin{align*}
        |(D(\mathcal{L})[\uu](\Phi),\ww_{K_j})| &=|(\Phi,D(\mathcal{L})[\uu](\ww_{K_j}))|\\
        &\leq \Vert\Phi\Vert_2\Vert D(\mathcal{L})[\uu](\ww_{K_j})\Vert_2.
    \end{align*}
Then since $\ww_{K_j}\rightharpoonup \ww$ in $L^2(\Omega;\mathbb{R}^d)$ it follows from the estimate above and the symmetry of $D(\mathcal{L})[\uu]$ that
\begin{align*}
        0=(D(\mathcal{L})[\uu](\Phi),\ww) = (\Phi,D(\mathcal{L})[\uu](\ww)),
    \end{align*}
for all test functions $\Phi$ so $D(\mathcal{L})[\uu](\ww)=0$ almost everywhere so $\ww\in Ker\{D(\LL)[\uu]\}$, hence $\ww=\boldsymbol{0}$,
by the hypothesis. 

Now it is shown that $\ww_{k_j}\rightarrow \boldsymbol{0}$ in the strong topology of $L^2(\Omega;\mathbb{R}^d)$, this will give the contradiction since $\Vert \ww_{K_j}\Vert_2=1$.  Recall that $D(\mathcal{L})[\uu](\ww_{K_j})=\mathbb{K}[\uu]\ww_{K_j}+\mathbb{A}[\uu]\ww_{K_j}$ where  $\mathbb{K}[\uu]$ is a compact operator on $L^2(\Omega;\mathbb{R}^d)$ and $\mathbb{A}[\uu]^T(\xx)=\mathbb{A}[\uu](\xx)$. Since $\ww_{K_j}$ goes weakly to zero we can pass to a subsequence $\mathbb{K}[\uu]\ww_{K_j}$ converging to zero the $L^2(\Omega;\mathbb{R}^d)$ norm. We write
\begin{align}\label{triangle}
   \Vert\mathbb{A}[\uu]\ww_{K_j}\Vert_2&=\Vert D(\mathcal{L})[\uu](\ww_{K_j})-\mathbb{K}[\uu]\ww_{K_j}\Vert_2 \nonumber\\
   &\leq \Vert D(\mathcal{L})[\uu](\ww_{K_j})\Vert_2+\Vert\mathbb{K}[\uu]\ww_{K_j}\Vert_2,
\end{align}
and take limits to see that $\lim_{K_j\rightarrow 0} \Vert\mathbb{A}[\uu]\ww_{K_j}\Vert_2=0$. The hypothesis and the symmetry of  $\mathbb{A}$ gives
\begin{align}\label{lowerbound}
&\gamma^2\Vert\ww_{K_j}\Vert_2^2\leq([\mathbb{A}[\uu]]^2\ww_{K_j},\ww_{K_j})=(\mathbb{A}[\uu]\ww_{K_j},\mathbb{A}[\uu]\ww_{K_j})=\Vert\mathbb{A}[\uu]\ww_{K_j}\Vert_2^2.
\end{align}
Applying \eqref{lowerbound} shows that $\lim_{K_j\rightarrow 0}\Vert\ww_{K_j}\Vert_2=0$ and we get the contradiction establishing the inequality \eqref{inequality}.
\end{proof}
Since $Ker\{D(\LL)[\uu]\}=\{\boldsymbol{0}\}$ we see that $D(\mathcal{L})[\uu]$ is one to one. To establish that $D(\mathcal{L})[\uu]$ is onto, note that inequality \eqref{inequality} implies that is a closed operator so  $D(\mathcal{L})[\uu](\overline{\mathcal{V}}^2)=\{\ww\in\overline{\mathcal{V}}^2\perp Ker\{D(\mathcal{L})[\uu]\}\}$.
So the inverse exists and is bounded from \eqref{inequality}.

The next step is to show that
$D(\mathcal{L})[\uu]^{-1}$ exists on $\mathcal{V}$ as a bounded linear map in the $L^\infty(\Omega;\mathbb{R}^d)$ norm.
Denote the kernel of the map $D(\LL)[\uu]$ on $\overline{\mathcal{V}}^2$  by $Ker\{D(\LL)[\uu](\overline{\mathcal{V}}^2)\}$ and the kernel of the map $D(\LL)[\uu]$ on ${\mathcal{V}}$  by $Ker\{D(\LL)[\uu]({\mathcal{V}})\}$. Note $Ker\{D(\LL)[\uu]({\mathcal{V}})\}\subset Ker\{D(\LL)[\uu](\overline{\mathcal{V}}^2)\}$
so $Ker\{D(\mathcal{L})[\uu](\mathcal{V})\}\}=0$. So $D(\LL)[\uu]$ is one to one map from $\mathcal{V}$ into itself. Next we show that $D(\LL)[\uu]$ is onto $\mathcal{V}$.

To simplify notation we write $\mathbb{K}[\uu]=\mathbb{K}$ and $\mathbb{A}[\uu]=\mathbb{A}$.
Observe that $\vv \in D(\LL)[\uu](\mathcal{V})$ implies $\vv \in D(\LL)[\uu](\overline{\mathcal{V}}^2)$. So from Theorem \ref{8}  there exists a unique $\ww \in \overline{\mathcal{V}}^2$ such that
\begin{align}\label{onto}
    \vv = D(\LL)[\uu]\ww  \text{ in } \overline{\mathcal{V}}^2.
\end{align}	
To establish surjectivity we show that $\norm{\ww}_{\infty} \le C\norm{\ww}_2+\norm{\vv}_\infty$ for some constant independent of $\ww$. 
Writing  $D(\LL)[\uu](\ww)=\mathbb{K}+\mathbb{A}$ gives
\begin{align*}
    \norm{\vv}_{\infty} 
    & = 
    \norm{D(\LL)[\uu]\ww}_{\infty} = \norm{\mathbb{K}(\ww)(\xx) + \A(\xx) \ww(\xx) }_{\infty}
\end{align*}	

Since $\A(\xx)$ is a self-adjoint matrix valued field, we have for all $\boldsymbol{\eta}\in\mathbb{R}^d$
\begin{align*}
    \boldsymbol{\eta}^T \A^2(\xx) \boldsymbol{\eta} \ge \gamma^2 \abs{\boldsymbol{\eta}}^2.
\end{align*}

Since all the eigenvalues of $\A$ {are non zero}, 
$\A$ is invertible. Therefore,
\begin{align}
    \label{eq:twocases}
    &\abs{\mathbb{K}(\ww)(\xx) + \A(\xx) \ww(\xx)}
     = \left[\left(\mathbb{K}(\ww)(\xx) +\A(\xx) \ww(\xx)\right) \cdot 
    \left(\mathbb{K}(\ww)(\xx) + \A(\xx) \ww(\xx)\right)\right]^{\frac{1}{2}} \nonumber \\
    & = \left[\A(\xx) \left( \A^{-1} \mathbb{K}(\ww)(\xx) + \ww(\xx) \right)  \cdot 
    \A(\xx) \left( \A^{-1} \mathbb{K}(\ww)(\xx) + \ww(\xx) \right) \right]^\half \nonumber\\
    & = \left[ \A^2(\xx) \left(\A^{-1}\mathbb{K}(\ww) (\xx) + \ww(\xx)\right) \cdot \left(\A^{-1}\mathbb{K}(\ww) (\xx) + \ww(\xx)\right) \right]^\half \nonumber	\\
    & \ge {|\gamma|} \abs{\A^{-1}\mathbb{K}(\ww) (\xx) + \ww(\xx)}.
\end{align}	

Applying the reverse triangle inequality gives
\begin{align}
\label{Linfinity-L 2}
    \norm{\vv}_\infty=\norm{\mathbb{K}(\ww) + \A \ww}_{\infty}
     \ge {|\gamma|} \norm{\A^{-1}\mathbb{K}(\ww)  + \ww}_{\infty}
     \ge {|\gamma|} \abs{\norm{\A^{-1}\mathbb{K}(\ww) }_{\infty} - \norm{\ww}_{\infty}}.
\end{align}	

Now we have two cases. Suppose first that
\begin{align}
    \label{eq:case1}
    \norm{\A^{-1} \mathbb{K}(\ww)}_{\infty} \ge \norm{\ww}_{\infty}.
\end{align}	
We apply the Cauchy-Schwartz inequality and change of integration variable to get
\begin{align*}
    \abs{\mathbb{K}(\ww)(\xx)} \le \int\limits_{D\cap H_\epsilon(\xx)}^{} \frac{C}{\epsilon^{d+1}\omega_3\abs{\yy - \xx}} \abs{\ww(\yy)} d\yy 
    \le 
   C \norm{\ww}_{L^2(\Omega;\mathbb{R}^d)},
\end{align*}	
so
\begin{align}\label{eq:crucial}
    \norm{\mathbb{K}(\ww)}_{\infty} \le C\norm{\ww}_{L^2(\Omega;\mathbb{R}^d)},
\end{align}	
and equation \eqref{eq:case1} implies
\begin{align}\label{upperL2}
    \norm{\ww}_{\infty} &\le \norm{\A^{-1} \mathbb{K}(\ww)}_{\infty} \le |\gamma|^{-1} \norm{\mathbb{K}(\ww)}_{\infty} \le C |\gamma|^{-1} \norm{\ww}_2,
\end{align}	
where the third inequality follows from \eqref{eq:crucial} and we see $\ww\in \mathcal{V}$.
If instead we have 
\begin{align*}
    \norm{\A^{-1} \mathbb{K}(\ww)}_{\infty} < \norm{\ww}_{\infty},
\end{align*}	
then from equation \eqref{Linfinity-L 2} we have
\begin{align*}
    \norm{\vv}_{\infty} \ge |\gamma| \left( \norm{\ww}_{\infty} - \norm{\A^{-1}\mathbb{K}(\ww)}_{\infty} \right) > 0.
\end{align*}	
So,
\begin{align}\label{case2}
    |\gamma| \norm{\ww}_{\infty} 
    \le \norm{\vv}_{\infty} &+ |\gamma| \norm{\A^{-1} \mathbb{K}(\ww)}_{\infty}
    \le \norm{\vv}_{\infty} + \norm{\mathbb{K}(\ww)}_{\infty}\\
   & \le \Vert\vv\Vert_{\infty} +C  \norm{\ww}_2,
\end{align}	
where the third inequality follows from \eqref{eq:crucial} and we see $\ww\in \mathcal{V}$. This establishes the surjectivity and we conclude $D(\LL)[\uu]^{-1}$ exists on $\mathcal{V}$.

Now we can easily go further using \eqref{upperL2} and \eqref{case2}  to show that the inverse is bounded.
%
To see this we continue the two cases \eqref{upperL2} and \eqref{case2} making use of Lemma \ref{8}.
Continuing inequality \eqref{upperL2}
we get 
\begin{align}\label{upperL22}
    \norm{\ww}_{\infty} &\le \norm{\A^{-1} \mathbb{K}(\ww)}_{\infty} \le {|\gamma|}^{-1} \norm{\mathbb{K}(\ww)}_{\infty} \le C {|\gamma|}^{-1} \norm{\ww}_2\nonumber\\
    &\le C {|\gamma|}^{-1}\Vert D(\LL)[\uu]\ww\Vert_2\le \sqrt{|\Omega|} C {|\gamma|}^{-1}\norm{D(\LL)[\uu]\ww}_{\infty},\nonumber
\end{align}	
where the second to last inequality follows from Lemma \ref{8},  and the first case delivers the desired inequality.
For the second case
\begin{align*}
    {|\gamma|} \norm{\ww}_{\infty} 
    \le \norm{\vv}_{\infty} &+ {|\gamma|} \norm{\A^{-1} \mathbb{K}(\ww)}_{\infty}
    \le \norm{\vv}_{\infty} + \norm{\mathbb{K}(\ww)}_{\infty}\nonumber\\
    &\le \norm{\vv}_{\infty} +C  \norm{\ww}_2\nonumber\\
    &\le \norm{D(\LL)[\uu]\ww}_{\infty} +C\norm{D(\LL)[\uu]\ww}_{\infty}\nonumber\\
    &=(1+C)\norm{D(\LL)[\uu]\ww}_{\infty}\nonumber,
\end{align*}	
where the last inequality follows as before  from Lemma \ref{8} and the second case delivers the desired inequality.
Thus we have shown $D(\LL)[\uu]^{-1}$ exists as a bounded linear transform over $\mathcal{V}$.
This proves \Cref{esistence of inverse in strength domain}.

\subsection{{Step 2 of proof of Theorem \ref{thm:stabinvI}}}
\label{sec:symmbounded}
In this step we show that $D(\LL)[\uu]$ satisfies all the hypotheses of theorem \ref{esistence of inverse in strength domain} 
 when $\uu$ is strictly inside the strength domain.
We begin by showing $D(\LL)[\uu]$ is symmetric and bounded in $L^2(\Omega;\mathbb{R}^d)$. 
\begin{lemma}\label{thm:symmet} Given $\uu\in\mathcal{V}$,
    $D(\LL)[\uu]$ is symmetric on $L^2(\Omega;\mathbb{R}^d)$.
\end{lemma}

\proof
To show that 
\begin{align*}
\int\limits_{\Omega}^{} D(\LL)[\uu](\ww) \cdot \vv d\xx
= \int\limits_{\Omega}^{} D(\LL)[\uu](\vv) \cdot \ww d\xx,
\end{align*}	
we appeal to the definition,
\begin{align}
    \label{eq:split}
    \begin{split}
-\int\limits_{\Omega}^{} D(\LL)[\uu](\ww) \cdot \vv d\xx
    & = \int\limits_{\Omega}^{} \int\limits_{\Omega}^{} \rho(\yy ,  \xx, S(\yy, \xx, \uu))   \ww(\yy) \cdot \ee_{\yy - \xx} \vv(\xx) \cdot \ee_{\yy - \xx} d\yy d\xx
    \\
    & - \int\limits_{\Omega}^{} \int\limits_{\Omega}^{} \rho(\yy ,  \xx, S(\yy, \xx, \uu))   \ww(\xx) \cdot \ee_{\yy - \xx} \vv(\xx) \cdot \ee_{\yy - \xx} d\yy d\xx.
    \end{split}
\end{align}	
Here consider the first term on the right hand side and switch the variable names of $\xx$ and $\yy$. Then
using the interchange symmetry \eqref{eq:interchange} of $\rho$ and $\ee_{\xx - \yy}\otimes \ee_{\xx - \yy}$  in the $\xx$ and $\yy$ variables, together with Fubini's theorem \eqref{eq:split} gives
\begin{align*}
    & \int\limits_{\Omega}^{} \int\limits_{\Omega}^{} \rho(\yy ,  \xx, S(\yy, \xx, \uu)) \ww(\yy) \cdot \ee_{\yy - \xx} \vv(\xx) \cdot \ee_{\yy - \xx} d\yy d\xx
    \\
    & = \int\limits_{\Omega}^{} \int\limits_{\Omega}^{} \rho(\yy ,  \xx, S(\yy, \xx, \uu))   \ww(\xx) \cdot \ee_{\yy - \xx} \vv(\yy) \cdot \ee_{\yy - \xx} d\yy d\xx
\end{align*}	
The result follows on combining the above with the second term on  the right hand side of \eqref{eq:split}.
\qed

Next we show $D(\LL)[\uu]$ is bounded. 
\begin{lemma}\label{bl2} Given $\uu\in\mathcal{V}$ the operator
    $D(\LL)[\uu]$ is bounded in $L^2(\Omega;\mathbb{R}^d)$, i.e., there exits $C > 0$ independent of $\epsilon$ and any $\ww \in L^2(\Omega;\mathbb{R}^d)$ such that
    \begin{align*}
	\norm{D(\LL)[\uu] \ww}_{2} \le\frac{ C}{\epsilon^2} \norm{\ww}_2.
    \end{align*}	
\end{lemma}
\proof
Write 
\begin{align*}
    & \norm{D(\LL)[\uu] \ww}_2^2 \\
    & = \int\limits_{\Omega}^{} \abs{\int\limits_{\Omega}^{} \rho(\yy, \xx, S(\yy, \xx, \uu)) \left( \ww(\yy) - \ww(\xx) \right) \cdot \ee_{\yy - \xx} \ee_{\yy - \xx} d\yy}^2 d\xx \\
    & \leq \int\limits_{\Omega}^{} \int\limits_{\Omega}^{} \int\limits_{\Omega}^{} \abs{  \rho(\yy, \xx, S(\yy, \xx, \uu)) \rho(\zz, \xx, S(\zz, \xx, \uu))  \left( \ww(\yy) - \ww(\xx) \right) \cdot \ee_{\yy - \xx}
    }\times
    \\
    \\ & \quad\quad\quad
    \abs{
    \left( \ww(\zz) - \ww(\xx) \right) \cdot \ee_{\zz - \xx}  } d\yy d\zz d\xx 
    \\
    & \le \int\limits_{\Omega}^{} \int\limits_{\Omega}^{} \int\limits_{\Omega}^{}  \rho(\yy, \xx, S(\yy, \xx, \uu)) \rho(\zz, \xx, S(\zz, \xx, \uu))  ( \abs{\ww(\yy)} + \abs{\ww(\xx)} ) ( \abs{\ww(\zz)} + \abs{\ww(\xx)} )  d\yy d\zz d\xx.
\end{align*}	

Expanding the product $( \abs{\ww(\yy)} + \abs{\ww(\xx)} ) ( \abs{\ww(\zz)} + \abs{\ww(\xx)} )$ we apply the inequality $ab \le \frac{1}{2}(a^2 + b^2)$ to the first term to get 
\begin{align*}
    & \int\limits_{\Omega}^{} \int\limits_{\Omega}^{} \int\limits_{\Omega}^{}  |\rho(\yy,\xx,S(\yy,\xx,\uu))||\rho(\zz,\xx,S(\zz,\xx,\uu))|   \abs{\ww(\yy)}  \abs{\ww(\zz)} d\yy d\zz d\xx \\
    & \le \int\limits_{\Omega}^{} \int\limits_{\Omega}^{} \int\limits_{\Omega}^{}  \frac{C\chi(\yy,\xx)}{\epsilon^{d+1} \omega_d\abs{\yy - \xx}}\frac{C\chi(\zz,\xx)}{\epsilon^{d+1} \omega_d\abs{\zz - \xx}} \abs{\ww(\yy)  }^2 d\yy d\zz d\xx\\
    & \le \frac{C}{\epsilon^2} \norm{\ww}_{2}^2,
\end{align*}	
where to pass to the last line we applied \eqref{rhoest}, made the change of variables $\zz-\xx=\epsilon\boldsymbol{\eta}$, $\yy-\xx=\epsilon\boldsymbol{\zeta}$ and switched limits of integration. Here the label $C$ denotes a constant independent of $\epsilon$.


The next terms are
\begin{align*}
    & \int\limits_{\Omega}^{} \int\limits_{\Omega}^{} \int\limits_{\Omega}^{}  \rho(\yy, \xx, S(\yy, \xx, \uu)) \rho(\zz, \xx, S(\zz, \xx, \uu))  \left(\abs{\ww(\yy)}  \abs{\ww(\xx)} + \abs{\ww(\xx)}  \abs{\ww(\zz)}\right) d\yy d\zz d\xx \\
    & = 2 \int\limits_{\Omega}^{} \int\limits_{\Omega}^{} \int\limits_{\Omega}^{}  \rho(\yy, \xx, S(\yy, \xx, \uu)) \rho (\zz, \xx, S(\zz, \xx, \uu)) \abs{\ww(\yy)}  \abs{\ww(\xx)} d\yy d\zz d\xx \\
    & = 2\int\limits_{\Omega}^{} \int\limits_{\Omega}^{} \rho(\yy, \xx, S(\yy, \xx, \uu))  \abs{\ww(\yy)}\abs{ \ww(\xx)}  \left(\int\limits_{\Omega}^{} \rho(\zz, \xx, S(\zz, \xx, \uu)) d\zz\right) d\yy d\xx \\
    &
    \le \frac{C}{\epsilon^2} \norm{\ww}_{2}^2
\end{align*}
where the last inequality follows from Young's integral inequality.
The final term
\begin{align*}
    &  \int\limits_{\Omega}^{} \int\limits_{\Omega}^{} \int\limits_{\Omega}^{}  \rho(\yy , \xx , S(\yy, \xx, \uu)) \rho(\zz , \xx, S(\zz, \xx, \uu))  \abs{\ww(\xx)}^2 d\yy d\zz d\xx \\
    & \le  \int\limits_{\Omega}^{}  \left( \int\limits_{\Omega}^{} \rho(\yy , \xx, S(\yy, \xx, \uu))   d\yy \right) \left( \int\limits_{\Omega}^{} \rho(\zz , \xx, S(\zz, \xx, \uu))d\zz \right) \abs{\ww(\xx)}^2 d\xx \\
    &  \le \frac{C}{\epsilon^2} \norm{\ww}_2^2 
\end{align*}	
The result follows by combining all 3 terms.
\qed
Theorem \ref{thm:symmetricv} follows immediately from Lemmas \ref{thm:symmet} and \ref{bl2}.

First note from \eqref{eq:positivity} that when the strain lies strictly inside the strength domain one has $\rho(\yy,\xx,S(\yy,\xx,\uu))>0$. The goal is to now show that $Ker\{D(\mathcal{L}[\uu]\}=\{\boldsymbol{0}\}$ when $\uu$ lies strictly inside the strength domain. 
We begin by writing an alternate formula for the quadratic form.
\begin{lemma}\label{lemma:quad}
\begin{align}
    \label{eq:du-u}
    \int\limits_{\Omega}^{} D(\LL)[\uu] \ww \cdot \ww d\xx = \frac{1}{2} \int\limits_{\Omega}^{}  \int\limits_{H_\epsilon(\xx)}^{} \rho(\yy , \xx, S(\yy, \xx, \uu)) \abs{ \left( \ww(\yy) - \ww(\xx) \right) \cdot  \ee_{\yy - \xx}}^2 d\yy d\xx
\end{align}	
\end{lemma}
From Lemma \ref{lemma:quad} and $\rho(\yy,\xx,S(\yy,\xx,\uu))>0$ it is seen that the quadratic form associated with $D(\LL)[\uu]$ is positive definite on $L^2(\Omega;\mathbb{R}^d)$.
\proof
\begin{align}\label{twoterms}
\begin{split}
    & -\int\limits_{\Omega}^{} D(\LL)[\uu] \ww \cdot \ww d\xx
    \\
    & = 
    \half \int\limits_{\Omega}^{} \int\limits_{H_\epsilon(\xx)}^{} \rho(\yy , \xx, S(\yy, \xx, \uu)) \left( \uu(\yy) - \ww(\xx) \right) \cdot \ee_{\yy - \xx} \ww(\xx) \cdot \ee_{\yy - \xx} d\yy d\xx
    \\
    &
    +\half \int\limits_{\Omega}^{} \int\limits_{H_\epsilon(\xx)}^{} \rho(\yy , \xx, S(\yy, \xx, \uu)) \left( \ww(\yy) - \ww(\xx) \right) \cdot \ee_{\yy - \xx} \ww(\xx) \cdot \ee_{\yy - \xx} d\yy d\xx
\end{split}
\end{align}	
Writing the second term we see that
\begin{align} \label{2ndterm}
     & \half \int\limits_{\Omega}^{} \int\limits_{\Omega}^{}  \rho(\yy , \xx, S(\yy, \xx, \uu)) \left( \ww(\yy) - \ww(\xx) \right) \cdot \ee_{\yy - \xx} \ww(\xx) \cdot \ee_{\yy - \xx} d\yy d\xx
    \nonumber\\
    & = -\half \int\limits_{\Omega}^{} \int\limits_{\Omega}^{}  \rho(\yy , \xx, S(\yy, \xx, \uu)) \left( \ww(\xx) - \ww(\yy) \right) \cdot \ee_{\xx - \yy} \ww(\xx) \cdot \ee_{\xx - \yy}  d\xx d\yy
    \nonumber\\
    & = -\half \int\limits_{\Omega}^{} \int\limits_{\Omega}^{}  \rho(\yy , \xx, S(\yy, \xx, \uu)) \left( \ww(\yy) - \ww(\xx) \right) \cdot \ee_{\yy - \xx} \ww(\yy) \cdot \ee_{\yy - \xx}  d\yy d\xx,
\end{align}	
where the second term follows from exchanging the order of integration, the third follows from rewriting $\ww(\yy)-\ww(\xx)$ the last follows 
a  relabeling inner and outer variables of integration  and the symmetry of $\rho$ and $\ee_{\xx - \yy}\otimes \ee_{\xx - \yy}$ in $\xx$ and $\yy$.
The Lemma follows substituting \eqref{2ndterm} into the second term of the right hand side of  \eqref{twoterms}.
\qed

The next lemma is an adaptation of Lemma 1  of Mengesha-Du \cite{MengeshaDuNonlocal14}.
\begin{lemma}
$
    \int\limits_{\Omega}^{} D(\LL)[\uu]\ww \cdot \ww d\xx = 0 \text{ if and only if } \ww \in \Pi.
    $
\label{lemma:rigid}
\end{lemma}

\proof
From Lemma \ref{lemma:quad} we have
\begin{align}\label{positivedef}
    \int\limits_{\Omega}^{} D(\LL)[\uu] \ww \cdot \ww d\xx = \frac{1}{2} \int\limits_{\Omega}^{}  \int\limits_{H_\epsilon(\xx)}^{} \rho(\yy , \xx, S(\yy, \xx, \uu)) \abs{ \left( \ww(\yy) - \ww(\xx) \right) \cdot  \ee_{\yy - \xx}}^2 d\yy d\xx
\end{align}	
Since $\rho(\yy,\xx,S(\yy,\xx,\uu))>0$ for $S\uu)$ strictly inside the strength domain the Lemma follows from \eqref{positivedef} and Lemma 1 of \cite{MengeshaDuNonlocal14}.
\qed
\noindent Since $\overline{\mathcal{V}}^2\cap\Pi=\{\boldsymbol{0}\}$ we conclude that $Ker\{D(\mathcal{L})[\uu]\}=\{\boldsymbol{0}\}$ when $\uu$ lies inside the strength domain.  Moreover since $\mathcal{V}\subset \overline{\mathcal{V}^2}$ and we conclude that 
    \begin{equation}\label{stab50infty}
        Ker\{D(\mathcal{L})[\uu]\}=0, \hbox{for all $\uu\in\mathcal{V}$}.
    \end{equation}

Now we show there exists $\gamma>0$ such that $\mathbb{A}[\uu]\geq \gamma\mathbb{I}$.
\begin{lemma}
    \label{lemma:positive-definite}
Given that $\uu$ is strictly inside the strength domain 
    there exists $\gamma > 0$ such that 
    for all $\xx \in \R^d$, $d = 2, 3$,
    \begin{align*}
	\gamma  \I \le \mathbb{A}(\xx)=\int\limits_{\Omega}^{} \rho(\yy , \xx, S(\yy, \xx, \uu)) \ee_{\yy - \xx} \otimes  \ee_{\yy - \xx} d\yy
    \end{align*}	
    where $\I \in \R^{d \times d}$ is the identity matrix.
\end{lemma}

\proof
When $\uu$ is strictly contained in the strength domain then one has a $\tilde{\gamma}>0$ for which
\begin{equation}\label{pos}
g''(\sqrt{\yy-\xx}S(\yy,\xx,\uu))>\tilde{\gamma}
\end{equation}
so for this case one has the estimate
\begin{align}\label{eq:rhoestlow}
    \rho(\yy,\xx,S(\yy,\xx,\uu)\geq \frac{\tilde{\gamma}\chi(\yy,\xx) J^\epsilon(|\yy - \xx|)}{\epsilon^{d+1} \omega_d\abs{\yy - \xx}}.
\end{align}

We proceed as in \cite{DuGunLehZho}, \cite{MengeshaDuNonlocal14} and for $\xx\in\Omega$, $\boldsymbol{\eta}\in\mathbb{S}^{d-1}$  we introduce the function $\Phi(\xx,\boldsymbol{\eta})$ defined by
\begin{align*}
    \Phi(\xx,\boldsymbol{\eta}) =\int\limits_{\Omega}^{} \frac{\chi(\yy)\chi_{\epsilon}(|\yy-\xx|)J^\epsilon(|\yy - \xx|)}{\epsilon^{d+1} \omega_d\abs{\yy - \xx}} |\ee_{\yy - \xx}\cdot\boldsymbol{\eta}|^2  d\yy.
\end{align*}
Since the domain $\Omega$ satisfies the interior cone condition (see Definition \ref{thm:exitence-of-perturbation})  we apply the insights of \cite{DuGunLehZho},  \cite{MengeshaDuNonlocal14} to see that for any $\xx,\boldsymbol{\eta}\in\Omega\times\mathbb{S}^{d-1}$
we have at most a zero measure subset in $C_{\lambda,\theta}(\xx,\ee_{\xx})$ which is perpendicular to $\boldsymbol{\eta}$, so
\begin{align*}
    \boldsymbol{\eta}^T\mathbb{A}(\xx)\boldsymbol{\eta}>\Phi(\xx,\boldsymbol{\eta})>0.
\end{align*}
When $\Phi(\xx,\boldsymbol{\eta})$ is continuous on $\xx,\boldsymbol{\eta}\in\Omega\times\mathbb{S}^{d-1}$ we can conclude the existence of a positive constant $C$ for which
\begin{align*}
    C=\min_{\xx,\boldsymbol{\eta}}\{\Phi(\xx,\boldsymbol{\eta})\}.
\end{align*}
Applying the result of \cite{MengeshaDuNonlocal14}  we can easily show that $\Phi(\xx,\boldsymbol{\eta})$ is continuous and from \eqref{eq:rhoestlow} we conclude that
\begin{align*}
    \boldsymbol{\eta}^T\mathbb{A}(\xx)\boldsymbol{\eta}>\gamma|\boldsymbol{\eta}|^2>0,
\end{align*}
where $\gamma=\tilde{\gamma}C$.
\qed

Lemma \ref{thm:symmet}, Lemma \ref{bl2}, equation \eqref{stab50infty}, and Lemma \ref{lemma:positive-definite},  show that the hypotheses of Theorem \ref{esistence of inverse in strength domain} are satisfied if $\uu$ is strictly contained in the strength domain so $D(\mathcal{L})[\uu]$  has a bounded inverse on $\mathcal{V}$.

\subsection{Energy minimization among deformations in the strength domain}
\label{sec:energyminimize}
We establish Theorem \ref{thm:Unique}.
It is shown  that if $\LL(\uu_0)=\bb_0$ and $\uu_0$ is strictly in the strength domain  then $E[\uu] \ge E[\uu_0]$ for fields $\uu$ inside the strength the domain. From construction $g(r)$ is convex for $r^e<r<r^c$ hence if $\uu$ has strain with $\sqrt{|\yy-\xx|}S(\yy,\xx,\uu)\in (r^e,r^c)$, then
\begin{align*}
g(\sqrt{|\yy-\xx|}S(\yy,\xx,\boldsymbol{r}(t)))\leq t g(\sqrt{|\yy-\xx|}S(\yy,\xx,\uu)+(1-t)g(\sqrt{|\yy-\xx|}S(\yy,\xx,\uu_0)), 
\end{align*}
where $0\leq t\leq 1$ and $\boldsymbol{r}(t)=t\uu+(1-t)\uu_0$. So for $E[\uu]]$ defined by \eqref{eq:total_energy} we have
\begin{align*}
    E[\boldsymbol{r}(t)]\leq t E[\uu]+(1-t)E[\uu_0],
\end{align*}
hence
\begin{align*}
   \frac{ E[\boldsymbol{r}(t)]-E[\uu_0]}{t}\leq E[\uu]-E[\uu_0],
\end{align*}
sending $t$ to zero gives
\begin{align*}
  \int_\Omega\,(\LL(\uu_0)-\bb_0)\cdot(\uu-\uu_0)\,d\xx \leq E[\uu]-E[\uu_0],
\end{align*}
and the claim follows noting that $\LL(\uu_0)=\bb_0$.

\section{A necessary condition for an inverse}
\label{sec:necessary}
We prove Theorem \ref{thm:ness}.  Suppose there is a set $\mathcal{F}\subset\Omega$ with nonzero measure for which $Ker\{\mathbb{A}(\xx)\}\not=\{\boldsymbol{0}\}$ on $\mathcal{F}$. So there is a measurable function $\mathbf{a}(\xx)$ taking values in $S^{d-1}$ on $\mathcal{F}$ and zero otherwise, with $\mathbb{A}(\xx)\mathbf{a}(\xx)=0$ on $\Omega$. We let $B(\xx_0,\delta)$ be the ball of radius $\delta$ centered at $\xx_0\in\mathcal{F}$ and form the function 
\begin{equation}
    \label{example}
    \uu^\delta_{\xx_0}(\xx)=\frac{\mathbf{a}(\xx)\chi_{\mathcal{F}}(\xx)\chi_{B(\xx_0,\delta)}(\xx)}{\sqrt{B(\xx_0,\delta)}},
\end{equation}
   where $\chi_{\mathcal{F}}$ and $\chi_{B(\xx_0,\delta)}$ are the indicator functions of the sets $\mathcal{F}$ and $B(\xx_0,\delta)$ respectively.
Choose $\delta$ small enough so that $B(\xx_0,\delta)\subset\Omega$ and for this choice we apply \eqref{rhoest} to get
\begin{align}
    \label{reducedD}
    |D(\LL)[\uu]\uu^\delta_{\xx_0}(\xx)|=&|\mathbb{K}\uu^\delta_{\xx_0}(\xx)|\\
    =&\left\vert\int\limits_{ \Omega}^{} \rho(\yy,\xx,S(\yy,\xx,\uu))\left(\uu^\delta_{\xx_0}(\yy)\cdot \ee_{\yy - \xx}\right)  \ee_{\yy - \xx}  d\yy \right\vert \nonumber\\
    \leq & \int\limits_{ \Omega}^{}\frac{C}{\epsilon^{d+1}\omega_d\abs{\yy - \xx}}\left\vert\left(\uu^\delta_{\xx_0}(\yy)\cdot \ee_{\yy - \xx}\right)  \ee_{\yy - \xx}\right\vert  d\yy\nonumber\\
    \leq & \tilde{C}\delta^{-\frac{d}{2}}\int\limits_{B(\xx_0,\delta)}^{}\frac{1}{\abs{\yy - \xx}}\left\vert\left(\mathbf{a}(\yy)\cdot \ee_{\yy - \xx}\right)  \ee_{\yy - \xx}\right\vert  d\yy\nonumber\\
    &\leq \tilde{C}\delta^{-\frac{d}{2}}\int\limits_{B(\xx_0,\delta)}^{}\frac{1}{\abs{\yy - \xx_0}}| d\yy\leq \tilde{C}\delta^{-\frac{1}{2}}\nonumber,
\end{align}
   here  $\tilde{C}$ denotes a constant independent of $\delta$.
   We arrive at the estimates
   \begin{equation}\label{broken}
   \Vert D(\LL)[\uu]\uu^\delta_{\xx_0}\Vert_\infty\leq\tilde{C}\delta^{\frac{1}{2}},\qquad \Vert\uu^\delta_{\xx_0}\Vert_\infty=\omega_d^{-\frac{1}{2}} \delta^{-\frac{d}{2}}.
   \end{equation}
   We summarize concluding for any sequence of positive numbers $K\rightarrow 0$ there is an appropriate $\vv_K=\uu^{\delta_K}_{\xx_0}/\Vert\uu^{\delta_K}_{\xx_0}\Vert_\infty$ for which $\Vert\vv_K\Vert_\infty=1$ and
   \begin{equation*}
       \Vert D(\LL)[\uu]\vv_K\Vert_\infty\leq K.
   \end{equation*}
   This shows that $D(\LL)[\uu]$ is not invertible on $Ran\{D(\LL)[\uu]\}$.

\section{Energy balance}
\label{energybalance}

In this section the energy balance given by theorem \ref{thm:Energy-load} is established. 
Since $D(\LL)[\uu]^{-1}$ exists for all $\uu\in \overline{B(\uu_0,R)}$  there exists a positive constant $K$ independent of $\uu\in B(\uu_0,R)$ for which
\begin{align}\label{linftyontou-0}
        \Vert D(\mathcal{L})[\uu](\ww)\Vert_\infty\geq K\Vert\ww\Vert_\infty, \forall \ww\in\,\mathcal{V}.
    \end{align}
Next we  show that $\uu_t$ is defined and continuous when the body force $\bb(t)$ is continuously differentiable in $t$. First suppose 
\begin{equation*}
\mathcal{L}\uu(t)=\bb(t) \hbox{ for $0<t<T$},   
\end{equation*}
and set $\Delta\uu(t)=\uu(t+\Delta t)-\uu(t)$. Note that $D(\mathcal{L}[\uu(t)]$ is the Fre\'chet derivative of $\mathcal{L}\uu(t)$ so
\begin{align}
    \frac{\bb(t+\Delta t)-\bb(t)}{\Delta t}=&\frac{\mathcal{L}(\uu(t)+\Delta\uu(t))-\mathcal{L}\uu(t)}{\Delta t}\nonumber\\
    =&D(\mathcal{L})[\uu(t)]\left(\frac{\Delta\uu}{\Delta t}\right) + \boldsymbol{\eta}(t)\boldsymbol{q}(t),\label{deriv}
\end{align}
where $|\boldsymbol{\eta}(t)|\leq C \Vert\frac{\Delta\uu}{\Delta t}\Vert_\infty$ and $|\boldsymbol{q}(t)|\leq C \Vert\Delta\uu\Vert_\infty$.
From \eqref{linftyontou-0}
\begin{equation}\label{unbound}
\left\Vert D(\mathcal{L})[\uu(t)]\left(\frac{\Delta\uu}{\Delta t}\right)\right\Vert_\infty\geq K\left\Vert\frac{\Delta\uu}{\Delta t}\right\Vert_\infty.
\end{equation}
Since $\uu(t)$ is continuous in $t$ we have $\Vert \Delta\uu\Vert_\infty\rightarrow 0$ as $\Delta t\rightarrow 0$,  so if
$\limsup_{\Delta t\rightarrow 0}\boldsymbol{\eta}(t)\boldsymbol{q}(t)$ is non zero then $\limsup_{\Delta t\rightarrow 0}\Vert \frac{\Delta\uu}{\Delta t}\Vert_\infty=\infty$ which from \eqref{unbound} and \eqref{deriv} implies $\bb_t=\infty$, which contradicts the assumed differentiability of $\bb(t)$. Hence $\limsup_{\Delta t\rightarrow 0}\boldsymbol{\eta}(t)\boldsymbol{q}(t)=0$. Inverting gives
\begin{align}\label{limzero}
   D(\mathcal{L})[\uu(t)]^{-1}\left(\frac{\bb(t+\Delta t)-\bb(t)}{\Delta t}\right)
    =\frac{\Delta\uu}{\Delta t} + D(\mathcal{L})[\uu(t)]^{-1}\boldsymbol{\eta}(t)\boldsymbol{q}(t),
\end{align}
and on taking limits noting $D(\mathcal{L}[\uu(t)]^{-1}$ is a bounded operator we get
\begin{align}\label{eq:limzero for partial uu}
   D(\mathcal{L})[\uu(t)]^{-1}\left(\frac{\partial \bb(t)}{\partial t}\right)
    =\lim_{\Delta t\rightarrow0}\frac{\Delta\uu}{\Delta t}:=\frac{\partial \uu(t)}{\partial t},
\end{align}
and $\uu_t$ exists and is continuous in $t$. 

Multiplying \eqref{eq:loadcontrol} by $\uu_t$ and integrating over $D$ gives
\begin{align*}
    \int_\Omega\,\LL[\uu(t)]\uu_t(t)\,d\xx = \frac{d}{dt}PD(\uu(t)) = \int_\Omega\bb(t)\cdot\uu_t(t)\,d\xx,
\end{align*}
and the energy balance Theorem \ref{thm:Energy-load} follows from time integration. 

{If instead we only know that $D(\LL)[\uu_0]^{-1}$ exists we show that we can choose $R$ smaller than or equal to that given in theorem \ref{thm:exitence-of-perturbation} such that $D(\LL)[\uu]^{-1}$ exists as a bounded linear functional on $\mathcal{V}$ for $\uu\in B(\uu_0,R)$. From Theorem \ref{thm:Frechetderiv} we have a $C$ independent of $\uu-\uu_0\in\mathcal{V}$ such that

\begin{align}\label{Lipschits}
    \normiii{D(\LL)[\uu]-D(\LL)[\uu_0]} &<  C||\uu-\uu_0||_\infty
\end{align}
Now from Banach's Lemma if
\begin{align}\label{Inverts}
    \normiii{D(\LL)[\uu]-D(\LL)[\uu_0]} &<  \normiii{D(\LL)[\uu_0]^{-1}}^{-1}
\end{align}
then $D(\LL)[\uu]^{-1}$ exists.  Taking $R$ small enough such that
\begin{equation}\label{Rchoice}
    CR\leq \normiii{D(\LL)[\uu_0]^{-1}}^{-1},
\end{equation}
we see that $D(\LL)[\uu]^{-1}$ exists for all $\uu\in \overline{B(\uu_0,R)}$.  On choosing  $R$ such that both \eqref{Rchoice}  and \eqref {eq:choice-R} hold  we get the conclusion of theorem \ref{thm:Energy-load}.
}

   \section{Conclusions}
\label{sec:conclusion} We have shown existence of quasistatic evolutions for nonlocal models exhibiting damage. The method shows existence of a solution as a rate independent evolution of critical points. It is shown that an evolution exists in the neighborhood of a nontrivial initial displacement field-body load pair $(\uu_0,\bb_0)$. For an appropriate body force $\bb_0$ the pair can be chosen such that $\uu_0$ lies in the strength domain and is a local minimizer of the PD energy.  
An energy balance law is established for the quasistatic load-controlled evolution.
{ More generally Theorem \ref{thm:exitence-of-perturbation} shows that there is a quasistatic evolution about and displacement-load pair $\uu_0,\bb_0\in \mathcal{V}$ for which $\LL[\uu_0] = \bb_0$ and  ${D(\LL)[\uu_0]^{-1}}$ exists and is bounded on on $\mathcal{V}$.}

A displacement controlled quasistatic evolution can also be addressed using methods introduced here. {Much of the theory developed here provides the mathematical foundations for the quasistatic fracture theory developed and implemented in the sequel \cite{BhattacharyaLiptonDiehl}.}

\bibliography{references}
\bibliographystyle{abbrv}

\end{document}